\documentclass{amsart}

\usepackage{amsmath, amssymb}
\input xy
\xyoption{all}
\begin{document}

\newtheorem{innercustomthm}{Theorem}
\newenvironment{customthm}[1]
  {\renewcommand\theinnercustomthm{#1}\innercustomthm}
  {\endinnercustomthm}
  
% These will be typeset in italics
\newtheorem{theorem}{Theorem}[section]
\newtheorem{proposition}[theorem]{Proposition}
\newtheorem{lemma}[theorem]{Lemma}
\newtheorem{corollary}[theorem]{Corollary}
\newtheorem{fact}[theorem]{Fact}

% These will be typeset in Roman
\theoremstyle{definition}
\newtheorem{definition}[theorem]{Definition}
\newtheorem{conjecture}[theorem]{Conjecture}
\newtheorem{notation}[theorem]{Notation}

\theoremstyle{remark}
\newtheorem{remark}[theorem]{Remark}
\newtheorem{example}[theorem]{Example}
\newtheorem{question}[theorem]{Question}

\numberwithin{equation}{section}

\def\id{\operatorname{id}}
\def\alg{\operatorname{alg}}
\def\Frac{\operatorname{Frac}}
\def\Const{\operatorname{Const}}
\def\spec{\operatorname{Spec}}
\def\span{\operatorname{span}}
\def\exc{\operatorname{Exc}}
\def\Div{\operatorname{Div}}
\def\cl{\operatorname{cl}}
\def\mer{\operatorname{mer}}
\def\trdeg{\operatorname{trdeg}}
\def\ord{\operatorname{ord}}
\def\rank{\operatorname{rank}}
\def\loc{\operatorname{loc}}
\def\kloc{\operatorname{K-loc}}
\def\acl{\operatorname{acl}}
\def\tp{\operatorname{tp}}
\def\CC{\mathbb C}
\def\PP{\mathbb P}

\title[Finiteness theorems on differential-algebraic hypersurfaces]{Finiteness theorems on hypersurfaces in partial differential-algebraic geometry}

\author{James Freitag}
\address{Department of Mathematics, UCLA, California, USA}
\email{freitag@math.ucla.edu}

\author{Rahim Moosa}
\address{Department of Pure Mathematics, University of Waterloo, Ontario, Canada}
\email{rmoosa@uwaterloo.ca}

\thanks{James Freitag was supported by an NSF Mathematical Sciences Postdoctoral Research Fellowship.}
\thanks{Rahim Moosa was supported by an NSERC Discovery Grant.}

\date{\today}

\begin{abstract}
Hrushovski's generalization and application of  [Jouanolou, ``Hypersurfaces solutions d'une \'equation de Pfaff analytique", {\em Mathematische Annalen}, 232 (3):239--245, 1978] is here refined and extended to the partial differential setting with possibly nonconstant coefficient fields.
In particular, it is shown that if $X$ is a differential-algebraic variety over a partial differential field~$F$ that is finitely generated over its constant field $F_0$, then there exists a dominant differential-rational map from $X$ to the constant points of an algebraic variety~$V$ over $F_0$, such that all but finitely many codimension one subvarieties of $X$ over~$F$ arise as pull-backs of algebraic subvarieties of~$V$ over $F_0$.
As an application, it is shown that the algebraic solutions to a first order algebraic differential equation over $\CC(t)$ are of bounded height, answering a question of Eremenko.
Two expected model-theoretic applications to $\operatorname{DCF}_{0,m}$ are also given: 1)~Lascar rank and Morley rank agree in dimension two, and 2) dimension one strongly minimal sets orthogonal to the constants are $\aleph_0$-categorical.
A detailed exposition of Hrushovski's original (unpublished) theorem is included, influenced by [Ghys, ``\`A propos d'un th\'eor\`eme de J.-P. Jouanolou concernant les feuilles ferm\'ees des feuilletages holomorphes", {\em Rend. Circ. Mat. Palermo (2)}, 49(1):175--180, 2000].
\end{abstract}

\maketitle

\setcounter{tocdepth}{1}
\tableofcontents

\section{Introduction}

\noindent
In a highly influential but unpublished manuscript from the mid nineteen nineties, Hrushovski showed that in the theory of differentially closed fields of characteristic zero an order one strongly minimal set  that is orthogonal to the constants must be $\aleph_0$-categorical.
His argument went via a certain finiteness theorem in differential-algebraic geometry:

\begin{theorem}[Hrushovski~\cite{hrushovski-jouanolou}]
\label{udi-ghj}
Suppose $X$ is a $\delta$-variety over $\mathbb C$ such that the constants of the $\delta$-function field  $\big(\mathbb C\langle X\rangle,\delta\big)$ is $\mathbb C$.
Then  $X$ has only finitely many codimension\footnote{Hrushovski uses the term {\em co-order}.} one $\delta$-subvarieties over $\mathbb C$.\footnote{When $X\subseteq\mathbb A^2$ is defined by $\delta(x)=P(x,y)$ and $\delta(y)=Q(x,y)$ where $P$ and $Q$ are polynomials over $\CC$, one recovers an old theorem of Darboux; see Singer's discussion and elementary proof of Darboux's Theorem in the appendix of~\cite{singer1992liouvillian}.}
\end{theorem}

\noindent
This theorem came up in the seemingly unrelated work of the second author and his collaborators on the Dixmier-Moeglin problem for Poisson algebras.
The following is the key step in the proof of one of the main results of that paper:

\begin{theorem}[Bell et al.~\cite{pdme}]
\label{bell-ghj}
If $R$ is a finitely generated $\mathbb C$-algebra equipped with (possibly noncommuting) $\mathbb C$-linear derivations $\Delta=\{\delta_1,\dots,\delta_m\}$, and having infinitely many height one prime $\Delta$-ideals, then there exists $f\in \Frac(R)\setminus\mathbb C$ with $\delta_i(f)=0$ for all $i=1,\dots,m$.
\end{theorem}

\noindent
When $m=1$ this is by standard methods seen to be a special case of Hrushovski's theorem, the so-called finite dimensional case when the transcendence degree of $\mathbb C\langle X\rangle$ over $\mathbb C$ is finite.
As the authors of~\cite{pdme} could not see how to extend Hrushovski's geometric proof to $m> 1$, an alternative algebraic argument was given in~\cite[$\S 6$]{pdme}.
One of the motivations for the present article is to extend Hrushovski's proof of Theorem~\ref{udi-ghj} to a setting that includes Theorem~\ref{bell-ghj}.
This is accomplished in~$\S$\ref{section-dvar}.

As it turns out, the right general setting is that of {\em $D$-varieties of type $(m,r)$}: algebraic varieties $V$ equipped with a subvariety $S$ of the $m$-fold fibred cartesian power of the tangent bundle, such that $S_a$ is an $r$-dimensional affine subspace of $(T_aV)^m$ for each $a\in V$.
To see what this has to do with Theorem~\ref{bell-ghj}, note that when $r=0$ the subvariety $S$ is given by $m$ regular sections to the tangent bundle, which in turn determines $m$ derivations on the co-ordinate ring of $V$.
On the other hand, if we specialise $m$ to $1$ we are in the setting of Theorem~\ref{udi-ghj} because, as Hrushovski shows in~\cite[Lemma 2.2]{hrushovski-jouanolou}, every $\delta$-variety together with its codimension one $\delta$-subvarieties is captured by a certain $D$-variety of type $(1,r)$ for some $r\geq 0$.
What we prove, for general $m$ and $r$, is the following.

\begin{customthm}{A}
Suppose $(V,S)$ is a $D$-variety of type $(m,r)$ such that for every $f\in\mathbb C(V)\setminus \mathbb C$ and general $a\in V$, it is not the case that $(d_af)^m$ vanishes on~$S_a$.
Then $(V,S)$ has only finitely many codimension one $D$-subvarieties over $\mathbb C$.
\end{customthm}

\noindent
This appears as Theorem~\ref{ghjD} below, following closely the proof of Theorem~\ref{udi-ghj}.
Theorem~\ref{udi-ghj} was itself obtained by Hrushovski as an application of a suitably generalized form of Jouanolou's~\cite{jouanolou} work on solutions to analytic Pfaffian equations.
We take this opportunity, in $\S$\ref{section-exposition}, to give a detailed exposition of Hrushovski's generalization, which we call the Jouanolou-Hrushovski-Ghys theorem because we utilise some simplifications appearing in Ghys'~\cite{ghys} improvement on Jouanolou's theorem.
In $\S$\ref{section-partial} we show how to extend the Jouanolou-Hrushovski-Ghys theorem to the case of arbitrary $m\geq 1$.
Theorem~A now follows exactly as it did for Hrushovski in~\cite{hrushovski-jouanolou}.

A second related motivation for this article was simply to extend Theorem~\ref{udi-ghj} to differential varieties in the partial case.
That is, to prove the theorem for differential subvarieties of a differential variety when $\delta$ is replaced by $m$ commuting derivations $\Delta:=\{\delta_1,\dots,\delta_m\}$.
It turns out that to deduce this from Theorem~A does require new work, and involves a seemingly new finiteness principle for partial differential equations. 
In $\S$\ref{section-dav} we use (partial) differential algebra and some combinatorics of initial sets to show that there is a bound on how many prolongations one has to take of a given $\Delta$-variety to capture all of its codimension one $\Delta$-subvarieties; this finiteness principle appears as Proposition~\ref{boundcodim1}. 
With this in place, we can use Theorem~A to prove the partial differential version of Theorem~\ref{udi-ghj}; it appears as Theorem~\ref{ghj-dav} below.
In fact, with a little more work we are able to both remove the assumption that the $\Delta$-variety $X$ is defined over the constants, and also formulate a version that does not assume that $X$ has no new $\Delta$-rational constants but rather is relative to whatever the constants of the $\Delta$-rational function field of $X$ are.
Here is the statement which appears as Corollary~\ref{ghj-dav-general} below.

\begin{customthm}{B}
Suppose $(K,\Delta)$ is a differentially closed field of characteristic zero in several commuting derivations, $F\subseteq\mathcal C$ is a subfield of the total constant field, $L$ is a finitely generated $\Delta$-field extension of $F$, and  $X\subseteq K^n$ is an $L$-irreducible $\Delta$-variety.
There exists an algebraic variety $V$ over the constants $L_0$ of $L$, and a dominant $\Delta$-rational map $f:X\to V(\mathcal C)$ over $L$, such that all but finitely many codimension one $L$-irreducible $\Delta$-subvarieties of $X$ are $L$-irreducible components of $\Delta$-subvarieties of the form $f^{-1}\big(W(\mathcal C)\big)$ where $W\subseteq V$ is an algebraic subvariety over $L_0$.
\end{customthm}

This theorem is particularly useful when applied to low dimensional differential varieties.
For example, if $X$ is one-dimensional then the codimension one subvarieties over $L$ arise from the $L^{\alg}$-points of $X$, and the theorem connects these points to the constant points of an algebraic curve.
As an application we are able to prove the existence of height bounds for solutions in $\CC(t)^{\alg}$ to first order algebraic differential equation.
The following appears as Theorem~\ref{boundheight} below.

\begin{customthm}{C}
Suppose $P\in\CC(t)[x,y]$ is a nonzero polynomial in two variables over the field of rational functions $\CC(t)$.
There exists $N=N(P)\in \mathbb N$ such that all solutions to $P(x,x')=0$ in $\big(\CC(t)^{\alg},\frac{d}{dt}\big)$ are of height $\leq N$.
\end{customthm}

\noindent
The height here is the function field absolute logarithmic height of Lang~\cite{lang}, which extends degree on rational solutions.
For rational solutions the existence of such a degree bound is a theorem of Eremenko~\cite{Erem}, where the extension to algebraic solutions was asked for.
A version of Theorem~C over a multivariate function field, and involving partial differentiation, can also be deduced from Theorem~B, see the discussion following Theorem~\ref{boundheight}.

We also present some model-theoretic applications of Theorem~B that arise from known consequences of 
Theorem~\ref{udi-ghj} which can now be extended to the partial and nonconstant coefficient field setting.
For example, as per Hrushovski's original motivation, we get that every one-dimensional strongly minimal set in $\operatorname{DCF}_{0,m}$ that is orthogonal to the constants is $\aleph_0$-categorical.
Another model-theoretic consequence of Theorem~B has to do with the comparing Lascar rank and Morley rank.
Hrushovski and Scanlon gave an example in~\cite{HrushovskiScanlon1999} of a differential algebraic variety of dimension five in which Lascar rank and Morley differ in $\operatorname{DCF}_0$.
They note that Marker and Pillay had an argument (also unpublished but communicated to us by the former) that used Theorem~\ref{udi-ghj} to show that for two-dimensional definable sets over the constants, Lascar rank and Morley rank agree.
%So, Hrushovski and Scanlon \cite{HrushovskiScanlon1999} ask if there is a theorem explaining the gap between two and five.
Using Theorem~B in place of Theorem~\ref{udi-ghj}, we extend the Marker-Pillay result to definable sets in $\operatorname{DCF}_{0,m}$ over arbitrary fields of definition.
%\footnote{Nagloo and Pillay \cite{Nagloo2011algebraic} claim that the family of Painlev\'e equations of types two through six each witness the non-equality of Lascar rank and Morley rank. In particular this would give an example of an order three variety for which Lascar rank is not equal to Morley rank. Combined with our result, this would fully resolve the question of Hrushovski and Scanlon. However, there appears to be a mistake in the computation of the ranks of these sets in \cite{Nagloo2011algebraic}.}
Both of these applications are given in~$\S$\ref{section-applications}, appearing as Theorems~\ref{lascarmorley} and~\ref{aleph0cat}.

We do not address in this paper the question of explicit bounds in the above finiteness theorems.
However, we do nothing here that is inherently ineffective, and as explicit bounds can be given for the original theorem of Jouanalou, one should in principle be able to give effective versions of our theorems too.

\medskip
{\em Acknowledgement}.
We would like to thank Matthias Aschenbrenner, Dave Marker, David McKinnon, and Tom Scanlon for a number of useful discussions.

\bigskip
\section{An exposition of the Jouanolou-Hrushovski-Ghys theorem}
\label{section-exposition}

\noindent
In this section we aim to give a detailed exposition of Hrushovski's~\cite{hrushovski-jouanolou}  unpublished generalization of Jouanolou's~\cite{jouanolou} theorem from $1$-forms to $p$-forms.
Our exposition is highly influenced by Ghys'~\cite{ghys} improvement on Jouanolou's theorem.

Let $X$ be a compact complex manifold.
By a {\em codimension $p$ holomorphic foliation on $X$} we will mean,
\begin{itemize}
\item
an open cover $(U_i)_{i\in I}$ of $X$,
\item
on each $U_i$ a $p$-fold wedge product of holomorphic $1$-forms,
$$0\neq \omega_i=\alpha_1\wedge\dots\wedge\alpha_p$$
and,
\item on each intersection $U_i\cap U_j$ a nonvanishing holomorphic function $g_{ij}$ such that $\omega_i=g_{ij}\omega_j$.
\end{itemize}
If we let $\mathcal L$ be the line bundle on $X$ defined by $(g_{ij})$, then the $\omega_i$'s determine a global holomorphic $p$-form on $X$ with values in $\mathcal L$, that is, $\omega\in H^0(X,\Omega^p\otimes\mathcal L)$.
Note that at each point $a\in X$, $\omega$ defines a codimension $p$ subspace of the tangent space $T_aX$, namely $\displaystyle W_a:= \bigcap_{\ell=1}^p\ker(\alpha_\ell)_a$ where $a\in U_i$ and $\omega_i=\alpha_1\wedge\dots\wedge\alpha_p$ on $U_i$.

By a {\em solution to $\omega=0$} we will mean a hypersurface $Y$ on $X$ whose tangent space at each point $a\in Y$ contains the subspace $W_a$.
In other words, possibly after refining the open cover, $Y$ is defined in $U_i$ by the vanishing of a holomorphic function $f_i$ on $U_i$ such that $(df_i)_a$ vanishes on $W_a$.
An equivalent formulation is that $\omega_i\wedge df_i\upharpoonright_{Y\cap U_i}= 0$.
Another equivalent formulation is that the meromorphic $(p+1)$-form $\omega_i\wedge\frac{df_i}{f_i}$ is in fact holomorphic on $U_i$.
To see this last equivalence note that, because $f_i$ generates the ideal of $Y\cap U_i$ in $\mathcal O(U_i)$,  $\omega_i\wedge df_i\upharpoonright_{Y\cap U_i} = 0$ if and only if $\omega_i\wedge df_i\upharpoonright_{Y\cap U_i}=f_i\eta$ for some holomorphic $(p+1)$-form $\eta$ on $U_i$.

A {\em meromorphic first integral} to $\omega$ is by definition a nonconstant meromorphic function on $X$ which is constant on the leaves of the foliation.
That is, an $f\in\mathbb C(X)\setminus\mathbb C$ such that $\omega\wedge df=0$.
Notice that if a mermorphic first integral to $\omega$ exists then $\omega=0$ has infinitely many solutions; namely the level sets of $f$.
The main theorem of this section is a converse to this observation.

\begin{theorem}[Jouanolou-Hrushovski-Ghys]
\label{ghj}
Suppose $X$ is a compact complex manifold.
If $\omega$ is a codimension $p$ holomorphic foliation on $X$ that does not admit a meromorphic first integral then $\omega=0$ has only finitely many solutions.
\end{theorem}

When $p=1$ and under some additional assumptions on $X$ (satisfied, for example, by smooth projective algebraic varieties), this is a theorem of Jouanolou appearing in the 1978 paper~\cite{jouanolou}.
With the same assumptions on $X$ as Jouanolou, and following his argumentation, Hrushovski proved the theorem for general $p$ in the unpublished manuscript~\cite{hrushovski-jouanolou} dating from the mid nineteen nineties.
A little later, Ghys~\cite{ghys} generalized Jouanolou's theorem in a different direction, removing the additional assumptions on $X$ and simplifying Jouanolou's argument, though only for $p=1$.
So while the theorem as stated here is formally new, it is obtained by simply combining Hrushovski's and Ghys' generalizations, and our purpose in presenting it here is entirely expository.

Let us denote by $\Div(X)$ the group of Weil divisors on $X$, and consider the logarithmic derivative map
$$d\log:\Div(X)\otimes\mathbb C\to H^1(X,\Omega_{\cl}^1)$$
where $\Omega_{\cl}^1$ is the sheaf of closed holomorphic $1$-forms on $X$.
To describe this map it suffices to define $d\log(Y)$ for irreducible hypersurfaces $Y$ on $X$, and then extend by $\mathbb C$-linearity.
If $Y$ is defined locally by $f_i=0$ then $d\log(Y)$ is the cocycle $\left(\frac{1}{g_{ij}}dg_{ij}\right)$ where $f_i=g_{ij}f_j$ on $U_i\cap U_j$.

There is a canonical injective $\mathbb C$-linear map
$$\xi:\ker(d\log)\to H^0(X,\Omega_{\cl,\mer}^1)/H^0(X,\Omega_{\cl}^1)$$
where $\Omega_{\cl,\mer}^1$ denotes the sheaf of closed meromorphic $1$-forms on $X$, that we now describe.
Suppose $x=\sum_\alpha\lambda_\alpha Y_\alpha\in\ker(d\log)$.
So if (after refining the cover) $Y_\alpha$ is given by $f_i^\alpha=0$ in $U_i$, and $f_i^\alpha=g^\alpha_{ij}f^\alpha_j$ on $U_i\cap U_j$, then the cocycle $\displaystyle \left(\sum_\alpha\lambda_\alpha\frac{dg_{ij}^\alpha}{g_{ij}^\alpha}\right)$ is a coboundary.
That is, $\sum_\alpha\lambda_\alpha\frac{dg_{ij}^\alpha}{g_{ij}^\alpha}=v_j-v_i$ on $U_i\cap U_j$, where $v_i$ and $v_j$ are closed holomorphic $1$-forms on $U_i$ and $U_j$ respectively.
It follows that
\begin{eqnarray*}
v_i+\sum_\alpha\lambda_\alpha\frac{df_i^\alpha}{f_i^\alpha}
&=&
v_i+\sum_\alpha\lambda_\alpha\frac{d(g_{ij}^\alpha f_j^\alpha)}{g_{ij}^\alpha f_j^\alpha}\\
&=&
v_i+\sum_\alpha\lambda_\alpha\frac{dg_{ij}^\alpha}{g_{ij}^\alpha} + \sum_\alpha\lambda_\alpha\frac{df_j^\alpha}{f_j^\alpha}\\
&=&
v_j+\sum_\alpha\lambda_\alpha\frac{df_j^\alpha}{f_j^\alpha}
\end{eqnarray*}
on $U_i\cap U_j$.
That is, $\left(v_i+\sum_\alpha\lambda_\alpha\frac{df_i^\alpha}{f_i^\alpha}\right)$ defines a global closed mermorphic $1$-form on $X$; this is what $\xi(x)$ is.
To see that $\xi$ is well-defined modulo $H^0(X,\Omega_{\cl}^1)$, a similar computation shows that if in the above construction we chose representatives $(\overline f_i^\alpha)$ and $(\overline v_i)$ instead, then writing $\overline f_i^\alpha=h_i^\alpha f_i^\alpha$ for some $h_i^\alpha$ a nowhere vanishing holomorphic function on $U_i$,
$$\left(\overline v_i+\sum_\alpha\lambda_\alpha\frac{d\overline f_i^\alpha}{\overline f_i^\alpha}\right) - \left(v_i+\sum_\alpha\lambda_\alpha\frac{df_i^\alpha}{f_i^\alpha}\right) = \overline v_i- v_i +\sum_\alpha\lambda_\alpha\frac{dh_i^\alpha}{h_i^\alpha}$$ which is a closed holomorphic $1$-form on $U_i$.
That they patch to produce an element of $H^0(X,\Omega_{\cl}^1)$ is a straightforward verification.

The map $\xi$ is injective because, as pointed out by Ghys~\cite[p.178]{ghys}, from the meromorphic $1$-form $\left(v_i+\sum_\alpha\lambda_\alpha\frac{df_i^\alpha}{f_i^\alpha}\right)$ we can recover the $Y^\alpha$ as the poles and the $\lambda^\alpha$ as the residues.

\begin{proof}[Proof of Theorem~\ref{ghj}]
Consider the $\mathbb C$-linear subspace of $\Div(X)\otimes\mathbb C$ spanned by solutions to $\omega=0$.
Namely,
$$\Div(\omega):=\left\{\sum_\alpha\lambda_\alpha Y_\alpha : \lambda_\alpha\in\mathbb C, Y_\alpha\text{ a solution to }\omega=0\right\}$$
In order to prove the Theorem we will assume that $\omega$ has no meromorphic first integral and show that $\Div(\omega)$ is a finite dimensional vector space.

Restricting $d\log:\Div(X)\otimes\mathbb C\to H^1(X,\Omega_{\cl}^1)$ to $\Div(\omega)$, and using the fact that $H^1(X,\Omega_{\cl}^1)$ is finite dimensional (as $X$ is compact), it suffices to show that
$$\Div_\circ(\omega):=\Div(\omega)\cap\ker(d\log)$$
is finite dimensional.

Next we can restrict the map $\xi$ constructed earlier and consider
$$\xi:\Div_\circ(\omega)\to H^0(X,\Omega_{\cl,\mer}^1)/H^0(X,\Omega_{\cl}^1)$$
Looking at that construction we see that if $x\in\Div_\circ(\omega)$ and $\xi_x\in H^0(X,\Omega_{\cl,\mer}^1)$ is a representative of $\xi(x)$, then $\omega\wedge\xi_x$, which is {\em a priori} in $H^0(X,\Omega^{p+1}_{\mer}\otimes\mathcal L)$, actually lands in $H^0(X,\Omega^{p+1}\otimes\mathcal L)$.
This follows from the fact that if $f_i$ defines a solution to $\omega=0$ in $U_i$ then by definition $\omega_i\wedge\frac{df_i}{f_i}$ is a {\em holomorphic} $(p+1)$-form on $U_i$.
So we obtain a $\mathbb C$-linear map
$\Div_\circ(\omega)\to H^0(X,\Omega^{p+1}\otimes\mathcal L)/\omega\wedge H^0(X,\Omega_{\cl}^1)$
given by taking $x$ to the class of $\omega\wedge\xi_x$.
The right-hand-side being finite dimensional, it suffices to show that the kernel of this map, let us denote it by $K$, is finite dimensional.

For each $x\in K$ we can choose a representative $\xi_x\in H^0(X,\Omega_{\cl,\mer}^1)$ of $\xi(x)$ such that $\omega\wedge\xi_x=0$.
Indeed, by choice of $K$, $\omega\wedge \xi_x=\omega\wedge\eta$ for some closed holomorphic $1$-form $\eta$, and we can replace $\xi_x$ with $\xi_x-\eta$.

By the injectivity of $\xi$, it will suffice to show that $\xi(K)$ is a finite dimensional $\mathbb C$-subspace of $H^0(X,\Omega_{\cl,\mer}^1)/H^0(X,\Omega_{\cl}^1)$.
This in turn reduces to showing that
$$\Xi:=\span_{\mathbb C}\{\xi_x:x\in K\}$$
is finite dimensional.
Note that $\Xi$ is a $\mathbb C$-subspace of the finite dimensional $\mathbb C(X)$-vector space $H^0(X,\Omega_{\cl,\mer}^1)=H^0(X,\Omega_{\cl}^1)\otimes_{\mathbb C}\mathbb C(X)$, where $\mathbb C(X)$ is the meromorphic function field of $X$.
By general exterior algebra (see Lemma~\ref{udi-lemma1.5}), it suffices to prove that for some $\ell\geq 1$,
$$B_\ell:=\span_{\mathbb C}\left\{\xi_{x_1}\wedge\dots\wedge \xi_{x_{\ell}}:x_1,\dots,x_\ell\in K\right\}$$
is a nontrivial finite dimensional $\mathbb C$-vector space; where the wedge product here is being taken in the sense of the $\mathbb C(X)$-vector space $H^0(X,\Omega_{\cl,\mer}^1)$.
We will work with $\ell$ equal to the dimension of the $\mathbb C(X)$-subspace generated by $\Xi$, and show that then $\dim_{\mathbb C}B_\ell=1$.
As we may assume that $K$ is not trivial (or else we are done), neither is $\Xi$, and so $\ell\geq 1$.

Let $\xi_{a_1},\dots,\xi_{a_\ell}$ be a basis for $\span_{\mathbb C(X)}\Xi$.
It follows by $\mathbb C(X)$-linear independence that
$\xi_{a_1}\wedge\dots\wedge\xi_{a_\ell}\neq 0$.
Moreover, for any $x_1,\dots,x_\ell\in K$, as each $\xi_{x_i}\in\span_{\mathbb C(X)}\{\xi_{a_1}.\dots,\xi_{a_\ell}\}$, straightforward exterior algebra shows that
\begin{equation*}
\xi_{x_1}\wedge\dots\wedge\xi_{x_\ell}=f\xi_{a_1}\wedge\dots\wedge\xi_{a_\ell}
\end{equation*}
for some $f\in\mathbb C(X)$.
Since we are working with closed $1$-forms here,
\begin{eqnarray*}
0
&=&
d(\xi_{x_1}\wedge\dots\wedge\xi_{x_\ell})\\
&=&
d(f\xi_{a_1}\wedge\dots\wedge\xi_{a_\ell})\\
&=&
df\wedge\xi_{a_1}\wedge\dots\wedge\xi_{a_\ell} +fd(\xi_{a_1}\wedge\dots\wedge\xi_{a_\ell})\\
&=&
df\wedge\xi_{a_1}\wedge\dots\wedge\xi_{a_\ell}
\end{eqnarray*}
But each $\omega\wedge\xi_{a_i}=0$, and so it follows from general exterior algebra (see Lemma~\ref{udi-lemma1.4}) that $\omega\wedge df=0$.
That is, as $\omega$ has no meromorphic first integral by assumption, $f$ must be a constant.
We have shown that $\dim_{\mathbb C}B_\ell=1$, and hence $\dim_{\mathbb C}\Xi$ is finite, as desired.
\end{proof}

\bigskip
\section{The partial case}
\label{section-partial}

\noindent
We would like to apply the Jouanolou-Hrushovski-Ghys theorem in the ``partial" setting where we replace the holomorphic tangent bundle by its $m$-fold direct sum.
No new ideas are required to make the proof go through, however there are some subtleties involved in setting things up correctly.

Let $T^mX\to X$ the direct sum of the holomorphic tangent bundle of $X$ with itself $m$~times.
As a complex manifold it is the $m$-fold fibred cartesian power of $TX$ over~$X$, so that for each $a\in X$, $(T^mX)_a=(T_aX)^m$.
We denote by $\Omega^{(1,m)}$ the sheaf of germs of holomorphic sections to the dual bundle of $T^mX\to X$.
One can of course identify this with $\bigoplus_{k=1}^m\Omega^1$, but for our purposes, namely for encoding families of subspaces of $(T_aX)^m$ as $a$ varies in $X$, we find it more convenient to work directly with $\Omega^{(1,m)}$.
We call it the sheaf of {\em holomorphic $m$-fold $1$-forms} on $X$.
The $m$-fold {\em $p$-forms} are then obtained by taking $p$th exterior powers, $\Omega^{(p,m)}:=\bigwedge^p\Omega^{(1,m)}$.
We will also consider $\Omega^{(p,m)}_{\mer}:=\Omega^{(p,m)}\otimes_{\mathbb C}\mathbb C(X)$, the sheaf of {\em meromorphic} $m$-fold $p$-forms on $X$.

For each $k=1,\dots, m$, the differential
$d_k:\mathcal O(U)\to \Omega^{(1,m)}(U)$
is given by
\begin{equation}
\label{dk}
(d_kf)_a(v_1,\dots,v_m):=(df)_a(v_k)
\end{equation}
for all $a\in U$ and $v_1,\dots, v_m\in T_aX$.
This can be extended to meromorphic functions in the same way.

An {\em $m$-fold holomorphic foliation on $X$ of codimension $p$} is a global holomorphic $m$-fold $p$-form on $X$ with values in a line bundle $\mathcal L$, say $\omega\in H^0(X,\Omega^{(p,m)}\otimes\mathcal L)$, such that for an open cover $(U_i)_{i\in I}$ of $X$ we have
$$0\neq \omega_i:=\omega\upharpoonright_{U_i}=\alpha_1\wedge\dots\wedge\alpha_p$$
where the $\alpha_\ell$ are holomorphic $m$-fold $1$-forms on $U_i$.

For each $a\in X$ we denote by $W_a\subseteq(T_aX)^m$ the codimension $p$ subspaces determined by $\omega_a$, namely, if $a\in U_i$ then $\displaystyle W_a:= \bigcap_{\ell=1}^p\ker(\alpha_\ell)_a$.

A {\em solution to $\omega=0$} is an irreducible hypersurface $Y$ on $X$ given locally by $f_i=0$ in $U_i$ such that $(d_kf_i)_a$ vanishes on $W_a$ for all $a\in Y\cap U_i$ and all $k=1,\dots,m$.
Equivalently, $(\omega_i\wedge d_kf_i)\upharpoonright_{Y\cap U_i}=0$ in $\Omega^{(p+1,m)}(U_i)$ for all~$k$.

A {\em meromorphic first integral} to $\omega$ is a nonconstant meromorphic function $f$ on $X$ with $\omega\wedge d_kf=0$ for all $k=1,\dots,m$.

\begin{theorem}
\label{partial-ghj}
Suppose $X$ is a compact complex manifold and $\omega$ is an $m$-fold holomorphic foliation on $X$ of codimension $p$.
If $\omega$ has no meromorphic first integral then $\omega=0$ has only finitely many solutions.
\end{theorem}

\begin{proof}
When $m=1$ this is Theorem~\ref{ghj}.
As before, let $\Div(\omega)$ be the $\mathbb C$-linear subspace of $\Div(X)\otimes\mathbb C$ spanned by hypersurfaces that are solutions to $\omega=0$.
Again it suffices to prove that $\Div_\circ(\omega):=\Div(\omega)\cap\ker(d\log)$ is finite dimensional.

Fixing $k=1,\dots,m$, let $\Omega^{(1,m)}_k$ denote the copy of $\Omega^1$ in $\Omega^{(1,m)}$ obtained by restricting to the $k$th factor.
That is, $f\in \Omega^1(U)$ is viewed as an element of $\Omega^{(1,m)}(U)$ by setting $f_a(v_1,\dots,v_m)=v_k$, for all $a\in U$ and $v_1,\dots,v_m\in T_aX$.
Under this embedding, the map $d_k:\mathcal O(U)\to \Omega_k^{(1,m)}(U)$ defined in~(\ref{dk}) above corresponds to the usual differential $d:\mathcal O(U)\to\Omega^1(U)$.
In the same way, we obtain copies $\Omega^{(1,m)}_{k,\cl}$ of $\Omega^1_{\cl}$, and $ \Omega^{(1,m)}_{k,\cl,\mer}$ of $\Omega^1_{\cl,\mer}$.
The injective $\mathbb C$-linear map $\xi:\ker(d\log)\to H^0(X,\Omega_{\cl,\mer}^1)/H^0(X,\Omega_{\cl}^1)$ constructed in the previous section now appears as
$\xi_k:\ker(d\log)\to H^0\big(X,\Omega^{(1,m)}_{k,\cl,\mer}\big)/H^0\big(X,\Omega^{(1,m)}_{k,\cl}\big)$.
So $\xi_k$ is defined just as $\xi$ was but using $d_k$ rather than $d$.

Note that if $(U_i,f_i=0)$ defines a solution to $\omega=0$, then for each $k=1,\dots,m$, the {\em a priori} meromorphic $m$-fold $(p+1)$-form $\omega_i\wedge\frac{d_kf_i}{f_i}$ is in fact holomorphic on~$U_i$.
From this it follows that we obtain a $\mathbb C$-linear map
$$\theta:\Div_\circ(\omega)
\longrightarrow
\bigoplus_{k=1}^m
H^0\big(X,\Omega^{(p+1,m)}\otimes\mathcal L\big)/\omega\wedge H^0\big(X,\Omega_{k,\cl}^{(1,m)}\big)$$
induced by $x\mapsto(\omega\wedge\xi_{1x},\dots,\omega\wedge\xi_{mx})$ where $\xi_{kx}$ is any representative of $\xi_k(x)$.
The right hand side being finite dimensional, we reduce to showing that $\ker\theta$ is finite dimensional.
We will do so by showing that its image under the injective map
$$\bar\xi:=(\xi_1,\dots,\xi_m):\ker(d\log)\longrightarrow
\bigoplus_{k=1}^m
H^0\big(X,\Omega^{(1,m)}_{k,\cl,\mer}\big)/H^0\big(X,\Omega^{(1,m)}_{k,\cl}\big)$$
is finite dimensional.

By definition of $\theta$, for any $x\in\ker\theta$, we can, and do, choose a representative $\xi_{kx}$ of $\xi_k(x)$ such that $\omega\wedge\xi_{kx}=0$.
Set $\bar\xi_x=(\xi_{1x},\dots,\xi_{mx})$.
It suffices to prove that $\Xi:=\{\bar\xi_x:x\in\ker\theta\}$ spans a finite dimensional $\mathbb C$-vector subspace of
$\bigoplus_{k=1}^m H^0\big(X,\Omega^{(1,m)}_{k,\cl,\mer}\big)$.
Note that as $\Omega^{(1,m)}$ is an internal direct sum of the $\Omega^{(1,m)}_k$s, we can view each $\bar\xi_x$ as an element of the finite dimensional $\mathbb C(X)$-vector space $H^0\big(X,\Omega^{(1,m)}_{\mer}\big)$.
So, using the same general facts about exterior algebra as in the proof of Theorem~\ref{ghj}, and letting $\ell=\dim_{\mathbb C(X)}\span_{\mathbb C(X)}\Xi$, we reduce to proving that
$\span_{\mathbb C}\{\bar\xi_{x_1}\wedge\dots\wedge\bar\xi_{x_\ell}:x_1,\dots,x_\ell\in\ker\theta\}$
is of dimension one, where the wedge product is taken in the sense of the $\mathbb C(X)$-vector space $H^0\big(X,\Omega^{(1,m)}_{\mer}\big)$.

Fix $a_1,\dots,a_\ell\in\ker\theta$ such that $(\bar\xi_{a_1},\dots,\bar\xi_{a_\ell})$ is a $\mathbb C(X)$-basis for $\span_{\mathbb C(X)}\Xi$.
Fix another $x_1,\dots,x_\ell\in\ker\theta$, and write $\bar\xi_{x_i}=\sum_{j=1}^\ell g_{ij}\bar\xi_{a_j}$ where the $g_{ij}\in\mathbb C(X)$.
Then for each fixed $k=1,\dots,m$ we have $\xi_{kx_i}=\sum_{j=1}^\ell g_{ij}\xi_{ka_j}$ too.
But this implies that
$\xi_{kx_1}\wedge\dots\wedge\xi_{kx_\ell}=f\xi_{ka_1}\wedge\dots\wedge\xi_{ka_\ell}$
where $f\in\mathbb C(X)$ depends only on the $g_{ij}$, and not on $k$.
So $\bar\xi_{x_1}\wedge\dots\wedge\bar\xi_{x_\ell}=f\bar\xi_{a_1}\wedge\dots\wedge\bar\xi_{a_\ell}$.
We are therefore done if we can show that $f\in\mathbb C$.

Fixing $k$, consider again the fact that $\xi_{kx_1}\wedge\dots\wedge\xi_{kx_\ell}=f\xi_{ka_1}\wedge\dots\wedge\xi_{ka_\ell}$.
We are working now with wedge products of forms in $\Omega^{(1,m)}_{k,\cl,\mer}$ which is an isomorphic copy of $\Omega^1_{\cl,\mer}$.
So the computation with closed meromorphic $1$-forms at the end of the proof of Theorem~\ref{ghj} shows that $d_kf\wedge\xi_{ka_1}\wedge\dots\wedge\xi_{ka_\ell}=0$.
Since $\omega\wedge\xi_{ka_i}=0$ for all $i$ by choice of representative, we get as before that $\omega\wedge d_kf=0$.
That this is true for all $k$ means that if $f$ were nonconstant then it would be a meromorphic first integral of $\omega$.
By assumption therefore, $f$ is a constant, as desired.
\end{proof}

\bigskip
\section{Hypersurfaces on $D$-varieties of type $(m,r)$}
\label{section-dvar}

\noindent
In~\cite{hrushovski-jouanolou} Hrushovski uses his generalization of Jouanolou's theorem to prove a finiteness theorem about codimension one differential-algebraic subvarieties.
We want to extend this theorem to the partial context, and in this section we first consider the {\em a priori} special case of $D$-varieties to which Hrushovski's arguments extend.

Fix $m\geq 1$ throughout.

By a $D$-variety we will mean something rather more general than usual:

\begin{definition}
Suppose $F$ is a field of characteristic zero.
An {\em algebraic $D$-variety of type $(m,r)$ over $F$} is an irreducible affine algebraic variety $V$ over~$F$ equipped with an irreducible closed subvariety $S\subseteq T^mV$ over $F$, such that $S_a$ is an $r$-dimensional affine subspace of $(T_aV)^m$ for all $a\in V$.

A {\em $D$-subvariety} is a closed irreducible subvariety $Y\subseteq V$ such that $S\upharpoonright_Y\subseteq T^mY$.
%\marginpar{\tiny Is this OK? If $Y$ is irreducible, does it follow that $S\upharpoonright_Y$ is irreducible?}
%It's OK: $T:=S\upharpoonright_Y$, $\pi:T\to Y$ is surjective. $T$ has a unique irreducible component $A$ that projects dominantly onto $Y$ by irreducibility of the generic fibre. Since the generic fibre of $A$ is of minimum dimension, and all the $S_y$ are of dimension $r$, we get that all the fibres of $A$ are $S_y$. So $A=\pi^{-1}\pi A$. On the other hand, as $\pi A$ is Zariski dense it contains a Zariski open set $U$. So $A$ contains the open subset $\pi^{-1}(U)$ of $T$, forcing $A$ to be all of $T$.

A rational function $f\in F(V)$ is a {\em $D$-constant} of $(V,S)$ if for general $a\in V$, $(df_a)^m:(T_aV)^m\to F^m$ vanishes on $S_a$.
\end{definition}

If $m=1$ and $r=0$ then $S$ is a section to the tangent bundle, and we recover what is usually called a ``$D$-variety'' in the literature.
Moreover, in that case, $S$ determines an $F$-linear derivation $\delta$ on the co-ordinate ring $F[V]$ which extends to $F(V)$, and a $D$-constant is simply a $\delta$-constant of that differential field.

\begin{theorem}
\label{ghjD}
Suppose $F$ is an algebraically closed field of characteristic zero and $(V,S)$ is a $D$-variety of type $(m,r)$ over $F$ with no $D$-constants in $F(V)\setminus F$.
Then $(V,S)$ has only finitely many codimension one $D$-subvarieties over $F$.
\end{theorem}

\begin{proof}
We basically need to verify that the arguments in~\cite[Proposition~2.3]{hrushovski-jouanolou} extend to this partial setting, though we give a self-contained exposition.

First note that it suffices to prove the theorem for $F=\mathbb C$.
Indeed, suppose $(V,S)$ is a counterexample to the theorem over $F$.
Let $F_0$ be a countable algebraically closed subfield of $F$ over which $(V,S)$ is defined, and over which $(V,S)$ has infinitely many codimension one $D$-subvarieties.
We may embed $F_0$ in $\mathbb C$.
As $V$ has no $D$-constants in $F_0(V)\setminus F_0$, and as $F_0$ is an algebraically closed subfield of $\mathbb C$, $(V,S)$ has no $D$-constants in $\mathbb C(V)\setminus \mathbb C$ either.
So $(V,S)$ is a counterexample over $\mathbb C$.

Let us consider the case when $r<m\dim V-1$.

Let $e\in V$ be generic.
Since $S_e$ is an affine subspace of $(T_eV)^m$, it generates an $(r+1)$-dimensional linear subspace of $(T_eV)^m$ over $\mathbb C(e)$, say $W_e$.
By assumption, $p:=\dim_{\mathbb C(e)}\big((T_eV)^m/W_e\big)>0$.
We can consider the $1$-dimensional space of $p$-forms on $(T_eV)^m/W_e$ as an algebraic variety $P$ over $\mathbb C(e)$.
It is an algebraic principal homogeneous space for $\mathbb G_a$ over $\mathbb C(e)$, and hence corresponds to a point in the Galois cohomological group $H^1(\mathcal G,\mathbb G_a)$ where $\mathcal G$ is the absolute Galois group of $\mathbb C(e)$.
The additive version of Hilbert's 90th tells us that $H^1(\mathcal G,\mathbb G_a)$ is trivial, so that $P$ is isomorphic to $\mathbb G_a$ over $\mathbb C(e)$.
So $P$ has a nonzero $\mathbb C(e)$-rational point, which we will denote by $\overline\beta$.
This will necessarily be of the form $\overline{\alpha_1}\wedge\dots\wedge\overline{\alpha_p}$, with $\bigcap_{\ell=1}^p\ker(\overline{\alpha_\ell})=0$.
Pulling back, we get a $p$-form $\beta=\alpha_1\wedge\dots\wedge\alpha_p$ on $(T_eV)^m$ with $\bigcap_{\ell=1}^p\ker(\alpha_\ell)=W_e$.
We thus obtain, over a nonempty Zariski open subset $U_0$ of the nonsingular locus of $V$, a nonzero regular section $\omega_0=\alpha^0_1\wedge\dots\wedge\alpha^0_p\in \Omega^{(p,m)}(U_0)$ such that $(\omega_0)_e=\beta$.

Let $X$ be a smooth projective closure of $U_0$.
So $\omega_0$ is rational on $X$, and by considering the line bundle corresponding to the polar divisor, $\omega_0$ extends to some $\omega\in H^0(X,\Omega^{(p,m)}\otimes\mathcal L)$, an $m$-fold regular foliation of codimension $p$ on $X$.
It is to this $\omega$ that we intend to apply Theorem~\ref{partial-ghj}.

We claim that $\omega$ admits no meromorphic (so rational) first integral.
Indeed, if $f\in\mathbb C(X)\setminus\mathbb C$ were such then $\omega_e\wedge (d_kf)_e=0$ which implies that $(d_kf)_e$ vanishes on $W_e\subseteq (T_eV)^m$, for all $k=1,\dots,m$.
But recall that $(d_kf)_e(v_1,\dots,v_m)=df_e(v_k)$ by definition.
So we have that $(df_e)^m$ vanishes on $W_e$ and hence on $S_e$.
Hence, for some nonempty Zariski open subset $U\subseteq V$, $(df_a)^m$ vanishes on $S_a$ for all $a\in U$.
That is, $f$ is a $D$-constant of $(V,S)$ that is not in $\mathbb C$, contradicting the assumption of the theorem.

By Theorem~\ref{partial-ghj}, it follows that $\omega=0$ has only finitely many solutions on $X$.
We now show that this will force there to be only finitely many codimension one $D$-subvarieties of $(V,S)$.

We work inside a sufficiently saturated model $(K,0,1,+,\times,,\delta_1,\dots,\delta_m)$ of the model companion of the theory of fields equipped with $m$ (not necessarily commuting) $\mathbb C$-linear derivations.
The existence and basic properties of this model companion are, we think, general knowledge. In any case, it is a special case of the theory of fields with free operators developed in~\cite{jetC}.
We let
$$\mathcal C:=\{x\in K:\delta_kx=0, k=1,\dots,m\}$$
denote the total constants of $K$.
The main reason for working in $K$ is that if $(V',S')$ is any $D$-variety over $\mathbb C$ then there is $a\in V'(K)$ such that $(a,\delta_1a,\dots,\delta_ma)$ is generic in $S'_a$ over $\mathbb C$; see for example~\cite[Theorem~4.6(III)]{jetC}, this is the so-called geometric axiom.
In particular, given a rational function $f\in \mathbb C(V')$, $f$ is a $D$-constant if and only if $f(a)\in\mathcal C$.
Indeed, this follows from the fact that $\delta_k\big(f(a)\big)=df_a(\delta_ka)$, and the genericity of $(\delta_1a,\dots,\delta_ma)$ in $S'_a$.

Suppose $Y$ is a codimension one $D$-subvariety of $(V,S)$ that intersects $U_0$.
Let $\overline Y$ be the Zariski closure of $Y\cap U_0$ in $X$.
We claim that $\overline Y$ is a solution to $\omega=0$.
That is, for a Zariski open cover $(U_i)_{i\in I}$ of $X$ with $\overline Y$ given in $U_i$ by the vanishing of a regular function $f_i$, we will show that $(\omega\wedge df_i)\upharpoonright_{\overline Y\cap U_i}=0$.

Since $Y$ is a $D$-subvariety there is $a\in Y(K)$ with $(a,\delta_1a,\dots,\delta_ma)$ a generic point of $S\upharpoonright_Y$ over $\mathbb C$.
In particular $a$ is generic in $\overline Y$, and so is contained in $U_0$ as well as each chart $U_i$.
It follows that $f_i(a)=0$ and so
$$(d_kf_i)_a(\delta_1a,\dots,\delta_ma)=(df_i)_a(\delta_k a)=\delta_k\big(f_i(a)\big)=0$$
for all $k=1,\dots,m$.
But $(\delta_1a,\dots,\delta_ma)$ is generic in $S_a$ over $\mathbb C(a)$, so we get that $(d_kf_i)_a$ vanishes on all of $S_a$, and as it is linear it must vanish on the subspace generated by $S_a$.
Note that as the $\alpha^0_\ell$ are regular $1$-forms on $U_0$ whose common kernel at the generic point $e$ is spanned by $S_e$, after shrinking $U_0$ further, we may assume that for all $x\in U_0$, $W_x:=\bigcap_{\ell=1}^p\ker(\alpha^0_\ell)_x$ is the $\mathbb C$-subspace of $(T_xV)^m$ spanned by $S_x$.
So $(d_kf_i)_a$ vanishes on all of $W_a$.
That is, $\omega_a\wedge(d_kf_i)_a=0$.
As $a$ is generic in $\overline Y\cap U_i$, we get $(\omega\wedge df_i)\upharpoonright_{\overline Y\cap U_i}=0$, as desired.

So we only get finitely many codimension one $D$-subvarieties of $(V,S)$ that intersect $U_0$.
As $V\setminus U_0$ is Zariski closed of codimension at least $1$, it can only contain at most finitely many codimension one subvarieties of $Y$.

It remains to consider the possibility that $r=m\dim V-1$.
(Note that when $r= m\dim V$ the theorem is vacuously true.)
For each $\gamma\in \mathbb C$, let $(\mathbb A^1,S_\gamma)$ denote the $D$-variety of type~$(m,0)$ where $S_{\gamma}$ is the graph of the section to the $m$-fold tangent bundle given by $a\mapsto (\gamma a,\dots,\gamma a)$.
Then $(V\times\mathbb A^1, S\times S_\gamma)$ is a $D$-variety of type $(m,r)$, and now $r<m\dim(V\times\mathbb A^1)-1$ so that the theorem is true of $(V\times\mathbb A^1, S\times S_\gamma)$.
Moreover, distinct codimension one $D$-subvarieties of $(V,S)$ give rise to distinct codimension one $D$-subvarieties of $(V\times\mathbb A^1, S\times S_\gamma)$ simply by taking the cartesian product with $(\mathbb A^1,S_\gamma)$.
So it suffices to show that $\gamma$ can be chosen in such a way that $(V\times\mathbb A^1, S\times S_\gamma)$ still has no nonconstant $D$-constants.

Suppose $(V\times\mathbb A^1, S\times S_\gamma)$ has a $D$-constant $g\in\mathbb C(V\times\mathbb A^1)\setminus\mathbb C$, and let us see what this implies about $\gamma$.
Using the geometric axiom, choose $(e,t)\in (V\times\mathbb A^1)(K)$ such that $(e,\delta_1e,\dots,\delta_me,t,\delta_1t,\dots,\delta_mt)$ is a generic point of $(S\times S_{\gamma})(K)$ over $\mathbb C$.
We claim that  $g(e,t)\not\in\mathbb C(e)^{\alg}$.
Otherwise, as $t$ is generic in $\mathbb A^1$ over $\mathbb C(e)$, it must be that $g(e,t)\in\mathbb C(e)$.
So $g(e,t)=h(e)$, and as $e$ is generic in $V$ over $\mathbb C$, we have that $h$ is a nonconstant $D$-constant of $(V,S)$, contradicting our assumption that such do not exist.
So $g(e,t)\notin\mathbb C(e)^{\alg}$.
It follows by Steinitz exchange that $t\in\mathbb C\big(e,g(e,t)\big)^{\operatorname{alg}}$. Since $g$ is a $D$-constant, what we have shown is that $t\in\mathcal C(e)^{\alg}$.

So, to show that $(V\times\mathbb A^1, S\times S_\gamma)$ has no nonconstant $D$-constants, it remains to verify that for some choice of $\gamma\in\mathbb C$, the set
$$E_\gamma:=\{x\in K:\delta_kx=\gamma x, k=1,\dots,m\}$$
has no point that is algebraic over $\mathcal C(e)$.
In fact, it follows from Fact~\ref{ostrowski} below, which as Hrushovski points out in~\cite{hrushovski-jouanolou} is a result of Kolchin's, that we can choose $\gamma\in \mathbb C$ such that $\delta_1x=\gamma x$ has no solution that is algebraic over $e$ together with the $\delta_1$-constants of $K$, and this is enough.\end{proof}

\begin{fact}[Kolchin~\cite{kolchin68}]
\label{ostrowski}
Suppose $(K,\delta)$ is a differential field with field of constants~$C$.
Let $F\subseteq K$ be a field extension of $C$ of finite transcendence degree.
Then the additive subgroup
$$\Gamma:=\{\gamma\in C:\delta x=\gamma x\text{ has a nonzero solution in F}\}$$
is of finite rank.
\end{fact}

\begin{proof}
Let $n>\trdeg(F/C)$, and suppose $\gamma_1,\dots,\gamma_n\in \Gamma$.
Then there are nonzero $a_1,\dots, a_n\in F$ such that $\frac{\delta(a_i)}{a_i}=\gamma_i\in C$, for all $i=1,\dots, n$.
By the choice of $n$ we have $a_1,\dots, a_n$ are algebraically dependent over $C$.
By what Kolchin calls the multiplicative analogue of Ostrowski's theorem in~\cite[page 1156]{kolchin68}, there are integers $e_1,\dots, e_n$, not all zero, such that $a_1^{e_1}a_2^{e_2}\cdots a_n^{e_n}\in C$.
Applying the logarithmic derivative $\frac{\delta x}{x}$ to this we get that $e_1\gamma_1+\cdots+e_n\gamma_n=0$.
\end{proof}

The following corollary appears as Theorem~6.1 of~\cite{pdme} but with an entirely different, longer and more algebraic, proof.
It was the key step in the proof of a weak (but optimal)  Dixmier-Moeglin equivalence for Poisson algebras.

\begin{corollary}
\label{pdme-version}
Let $R$ be a finitely generated integral $\mathbb C$-algebra equipped with $\mathbb C$-linear derivations $\delta_1,\dots,\delta_m$.
If there are infinitely many	height one prime differential ideals then there exists $f \in\Frac(R)\setminus\mathbb C$ with $\delta_i(f)=0$ for all $i=1,...,m$.
\end{corollary}

\begin{proof}
This is precisely the algebraic formulation of Theorem~\ref{ghjD} when $r=0$, with $V$ the affine algebraic variety whose co-ordinate ring is $R$ and $S$ the image of the regular section to $T^mV\to V$ induced by the derivations $\delta_1,\dots,\delta_m$ on $R$.
\end{proof}

\bigskip
\section{Hypersurfaces on differential-algebraic varieties}
\label{section-dav}
\noindent
We now turn our attention to partial differential-algebraic varieties in the context of $m$ commuting derivations, $\Delta=\{\delta_1,\dots,\delta_m\}$.
These can be viewed indirectly as $D$-varieties.
However the passage from $\Delta$-varieties to $D$-varieties involves taking a sufficiently long prolongation, and so to apply Theorem~\ref{ghjD} to this context will require proving there is a bound on how far one has to go to capture all the codimension one $\Delta$-subvarieties.
This is done in $\S$\ref{subsection-cdb}.
We then prove our main result: $\Delta$-varieties over the constants with no nonconstant $\Delta$-constant $\Delta$-rational functions have only finitely many codimension one $\Delta$-subvarieties (Theorem~\ref{ghj-dav}).
Finally, we deduce a version that makes no assumption on the $\Delta$-constant $\Delta$-rational functions and extends to arbitrary finitely generated $\Delta$-fields of definition (Corollary~\ref{ghj-dav-general}).

\medskip
\subsection{A review of differential-algebraic geometry}
\label{subsection-rdag}
This is meant primarily to fix notation.
Ours will be more or less standard, so the reader familiar with the subject can safely skip to the next section.
For further details on these preliminaries we suggest Chapters~I and IV of~\cite{kolchin73}.

Let $\Delta=\{\delta_1,\dots,\delta_m\}$ be the commuting {\em derivations}, and
$$\Theta:=\{\delta_m^{e_m}\cdots \delta_1^{e_1}:e_1,\dots,e_m\in\mathbb N\}$$
the corresponding {\em derivatives}.
The {\em order} of $\delta_m^{e_m}\cdots \delta_1^{e_1}$ is $e_1+\cdots+e_m$.
For $u=(u_1,\dots,u_n)$ a tuple of indeterminates, the set of {\em algebraic indeterminates} is
$\Theta u:=\{\theta u_i:\, 1\leq i\leq n, \, \theta\in \Theta\}$.
By the {\em order} of an algebraic indeterminate $\theta u_i$ we mean the order of $\theta$.
There is a canonical ranking on $\Theta u$ where
$\delta_m^{e_m}\cdots \delta_1^{e_1}u_i< \delta_m^{r_m}\cdots \delta_1^{r_1}u_j$ means that $\left(\sum e_k,i,e_m,\dots,e_1\right)<\left(\sum r_k,j,r_m,\dots,r_1\right)$ in the lexicographic order.

Suppose $(F,\Delta)$ is a partial differential field of characteristic zero.
We denote by $F\{u\}$ the $\Delta$-ring of $\Delta$-polynomials over $F$, and by $F\langle u\rangle$ its fraction field, the $\Delta$-field of $\Delta$-rational functions.
So the underlying $F$-algebra structure on $F\{u\}$ is that of the polynomial ring $F[\Theta u]$.
Let $f\in F\{u\}\setminus F$.
The \emph{leader} of~$f$, $u_f$, is the highest ranking algebraic indeterminate that appears in $f$.
The \emph{order} of~$f$ is the order of its leader.
The \emph{leading degree} of~$f$, $d_f$, is the degree of $u_f$ in $f$. 
The \emph{rank} of~$f$ is the pair $(u_f,d_f)$, and the set of ranks is ordered lexicographically.
By convention, an element of $F$ has lower rank than all the elements of $F\{x\}\setminus F$.
The \emph{separant} of $f$, $S_f$, is the formal partial derivative of $f$ with respect to $u_f$.
%The \emph{initial} of $f$, $I_f$, is the leading coefficient of $f$ when viewed as a polynomial in $u_f$.
Note that $S_f$ has lower rank than $f$.
%Given a finite subset $\Lambda\subset R\{x\}$, we set $H_\Lambda:=\prod_{f\in \Lambda}I_fS_f$.

The ranking on $\Delta$-polynomials is extended to finite sets of $\Delta$-polynomials as follows:
Writing finite sets of differential polynomial in nondecreasing order by rank, 
define $\{g_1,\dots,g_r\}<\{f_1,\dots,f_s\}$ to mean that either there is $i\leq r,s$ such that $\rank(g_j)=\rank(f_j)$ for $j<i$ and $\rank(g_i)<\rank(f_i)$, or $r>s$ and $\rank(g_j)=\rank(f_j)$ for $j\leq s$.

Suppose $\Lambda$ is a subset of $ F\{x\}\setminus F$.
the set $\Lambda$ is said to be \emph{autoreduced} if for each $f\neq g$ in $A$, no proper derivative of $u_f$ appears in $g$, and if $u_f$ appears at all in $g$ then it does so with strictly smaller degree.
Autoreduced sets are finite.

A {\em $\Delta$-ideal} of $F\{u\}$ is an ideal that is preserved by $\delta_1,\dots,\delta_m$.
If $I\subset F\{u\}$ is a prime $\Delta$-ideal then a {\em characteristic set} $\Lambda$ for $I$ is a minimal autoreduced subset of~$I$.
Prime $\Delta$-ideals are determined by their characteristic sets.

We will be concerned {\em $\Delta$-varieties}, namely sets of solutions to systems of $\Delta$-polynomials.
While this can be done at various levels of generality and abstraction, we will work essentially set-theoretically, fixing a sufficiently saturated ambient differentially closed field $(K,\Delta)$ and identify $\Delta$-algebraic varieties with their $K$-points.
That is we are considering the Kolchin topology on various cartesian powers of $K$.
This is a noetherian topology.

If $X\subseteq K^n$ is a $\Delta$-variety defined over $F$ then
$$I_\Delta(X):=\big\{f\in F\{u\} : f(x)=0\text{ for all }x\in X\big\}$$
is the {\em $\Delta$-ideal of $X$}.
When $X$ is $F$-irreducible it is a prime $\Delta$-ideal, and we denote by $F\langle X\rangle$ the {\em $\Delta$-rational function field} of $X$, i.e., the fraction field of $F\{u\}/I_\Delta(X)$ with the (unique) extension of the $\Delta$-field structure.

Given $c=(c_1,\dots,c_n)\in K^n$, by the {\em Kolchin locus of $c$ over $F$} we mean the smallest Kolchin closed dubset of $K^n$ over $F$ that contains $c$.
We will denote this by $\kloc(c/F)$, and the usual Zariski locus by $\loc(c/F)$.
If $X\subseteq K^n$ is a $\Delta$-variety defined over $F$ then by a {\em generic } point in $X$ over $F$ we mean $c\in X$ such that $X=\kloc(c/F)$.
Note that $F\langle X\rangle=F\langle c\rangle$, that is, the $\Delta$-rational function field over $F$ is generated over $F$ as a $\Delta$-field by a generic point.

\medskip

{\em For the remainder of this section we work in a fixed sufficiently saturated differentially closed field $(K,\Delta)$ of characteristic zero, with field of $\Delta$-constants $\mathcal C$.
We also fix a small $\Delta$-subfield $F\subseteq K$ which will serve as our field of defintion.}

\medskip
\subsection{Dimension and transcendence index sets}
\label{subsection-c1dav}
An important technique in the study of $\Delta$-varieties is to view them as proalgebraic varieties in the following sense.
For each $t<\omega$, and $c=(c_1,\dots,c_n)\in K^n$, let
$$\nabla_tc:=(\theta c_i: i=1,\dots,n, \ \theta\in\Theta\text{ of order}\leq t)$$
indexed with respect to the canonical ordering on $\Theta u$.
Then $\kloc(c/F)$ is determined by $\big(\loc(\nabla_tc/F):t<\omega\big)$, which is a directed system of algebraic varieties under the natural co-ordinate projections $\loc(\nabla_{t+1}c/F)\to\loc(\nabla_tc/F)$.

Suppose $X\subseteq K^n$ is an $F$-irreducible $\Delta$-variety, and $c\in X$ is generic over $F$.
By the {\em dimension function of $X$} we mean the sequence of natural numbers
$$\big(\trdeg_FF(\nabla_tc):t<\omega\big).$$
In working with these dimensions Kolchin's description of an explicit transcendence bases for $F(\nabla_tc)$ over $F$ will be very useful.
We first introduce some multi-index notation.

\begin{notation}
Regard
$\mathbb N^m \times \{1, \ldots , n\}$
as a partial order where
$$(r_1, \ldots , r_m , i) \leq (s_1, \ldots , s_m , j)$$
means that $i=j$ and $r_k \leq s_k$ for each $k=1,\dots,m$.
For $r=(r_1,\dots,r_m,i)\in\mathbb N^m \times \{1, \ldots , n\}$, $B\subset\mathbb N^m \times \{1, \ldots , n\}$, $x=(x_1,\dots,x_n)\in K^n$, and $t<\omega$, we set
\begin{itemize}
\item
$|r|:=r_1+\cdots+r_m$,
\item
$rx:=\delta_1^{r_1}\cdots\delta_m^{r_m}x_i$,
\item
$Bx:=(rx:r\in B)$,
viewed as a sequence of elements in $K$ indexed by $B$,
\item
$B_t:=\left\{(r_1,\dots,r_m,j)\in B:|r|\leq t\right\}$.
\end{itemize}
So in this notation if $r\geq s$ then $rx$ is a derivative of $sx$.
\end{notation}

The following fact is established in the proof of Theorem~6, Chapter II.12 of~\cite{kolchin73}.

\begin{fact}
\label{basis}
Suppose $X\subseteq K^n$ is an $F$-irreducible $\Delta$-variety, $c\in X$ is generic over $F$, and $\Lambda$ is a characteristic set for $I_\Delta(X)$.
Let $E$ denote the set of all points $(e_1,\dots,e_m,j)\in\mathbb N^m\times\{1,\dots,n\}$ such that $\delta_1^{e_1}\dots\delta_m^{_m}u_j$ is a leader of an element of~$\Lambda$, and $B$ the set of all points in $\mathbb N^m\times\{1,\dots,n\}$ that do not lie above any element of $E$.
Then, for all $t<\omega$, $B_tc$ is a transcendence basis for $F(\nabla_tc)$ over $F$.
\end{fact}

We will therefore call the set $B\subseteq \mathbb N^m \times \{1, \ldots , n\}$ appearing in Fact~\ref{basis} a {\em transcendence index set} for $c$ over $F$.
Note that $B$ is an {\em initial} set; it is a subset of $\mathbb N^m \times \{1, \ldots , n\}$ that is closed downward in the partial ordering.

The following lemma points out that the transcendence basis found in Fact~\ref{basis} is actually a {\em linear} basis over $\Delta$-rational functions of lower order.

\begin{lemma}
\label{affine}
Let $t$ be strictly greater than the order of every element of $\Lambda$.
Then $F(\nabla_t c)\subseteq\operatorname{span}_{F(\nabla_{t-1} c)}(B_tc)$.
\end{lemma}

\begin{proof}
We will use the following well known fact about $\Delta$-polynomials that can be verified easily by induction on the order.
Recall that $u=(u_1,\dots,u_n)$ are our $\Delta$-indeterminates.
\begin{itemize}
\item[($*$)]Suppose $f\in F\{u\}$ is a $\Delta$-polynomial of order $t$ and $\theta\in\Theta$ is a derivative of order $s>0$.
Let $w_1,\dots,w_p$ be the algebraic indeterminates of order $t$ appearing in $f$.
Then $\theta f$ is a degree one polynomial in $\theta w_1,\dots,\theta w_p$ with coefficients of order $<t+s$.
Moreover, if $w_1$ is the leader of $f$ then $\theta w_1$ is the leader of $\theta f$ and appears with coefficient $S_f$, the separant of $f$.
\end{itemize}

We prove by induction on the rank of $ru$, for $|r|=t$, that $rc\in \operatorname{span}_{F(\nabla_{t-1} c)}(B_tc)$.

If $r\in B_t$ there is nothing to prove.
So assume $r\notin B_t$, and suppose $ru=\theta u_j$.
Then there are derivatives $\theta _1,\theta_2$ such that  $ru=\theta_2\theta_1u_j$ and $\theta_1u_j$ is the leader of some $f\in\Lambda$.
Note that $\ord(\theta_2)>0$ since $t$ is greater than the order of all elements of $\Lambda$.
Let $w_1=\theta_1u_j,w_2,\dots,w_p$ be the algebraic indeterminates of order $\ord(f)$ that appear in $f$.
By~($*$), $\theta_2 f$ is of degree one in $\theta_2w_1,\dots,\theta_2w_p$ with coefficients of order $<\ord(f)+\ord(\theta_2)=t$.
Moreover, $ru=\theta_2w_1$ is the leader of $\theta_2 f$, and appears with coefficient $S_f$.
Therefore $ru$ is an $F[\nabla_{t-1}u,\frac{1}{S_f}]$-linear combination of $\{1,\theta_2w_2,\dots,\theta_2w_p\}$.
Since $c$ is generic in $X$ and $\Lambda$ is a characteristic set, $S_f(c)\neq 0$, and hence $rc\in\operatorname{span}_{F(\nabla_{t-1} c)}\{1,\theta_2w_2(c),\dots,\theta_2w_p(c)\}$.
Now $1\in \operatorname{span}_{F(\nabla_{t-1} c)}(B_tc)$ trivially.
On the other hand, each $\theta_2w_k$ for $k=2,\dots,p$, is of order $t$ and of rank strictly less than $ru$.
Hence by the induction hypothesis, each $\theta_2w_k(c)\in\operatorname{span}_{F(\nabla_{t-1} c)}(B_tc)$, completing the proof.

Note that this deals also with the base case of the induction, since if $ru$ is of minimal rank among order $t$ algebraic indeterminates, then $p$ must be $1$ in the above argument.
\end{proof}

Here is another property of transcendence index sets that will be useful.

\begin{lemma}
\label{additivity}
Given finite tuples $a$ and $b$, let $B\subseteq \mathbb N^m\times\{1,\dots,n\}$ be a transendence index set for $b$ over $F\langle a\rangle$.
There exists a natural number $N$, such that for all $t\geq 0$, $\nabla_tb\subseteq F(\nabla_{t+N}a,B_tb)^{\alg}$.
\end{lemma}

\begin{proof}
Let $\Lambda\subset F\langle a\rangle\{u\}$ be the characteristic set for $I_\Delta(b/F\langle a\rangle)$ corresponding to~$B$.
Let $\ell$ be an upper bound on the order of the elements of $\Lambda$.
If $f\in\Lambda$ then the coefficients of $f$ are $\Delta$-rational functions in $a$ over $F$.
Let $N$ be such that each of these coefficients, as $f$ ranges in $\Lambda$, can be written as a fraction of $\Delta$-polynomials in $a$ over $F$ of order $\leq N$.
We will show that this $\ell$ and $N$ work.

Let $t\geq\ell$.
We prove by induction on the rank of $ru$, for $r\in \mathbb N^m\times\{1,\dots,n\}$ with $|r|\leq t$, that $rb\in F(\nabla_{t+N}a,B_tb)^{\alg}$.
If $r\in B_t$ there is nothing to prove.
So assume $r\notin B_t$.
Then there is a derivative $\theta'$ such that  $ru$ is the leader of $\theta'f$ for some $f\in\Lambda$.
Moreover, when $\theta'f$ is viewed as a polynomial in $ru$ the leading coefficient is $S_f$, this is by~($*$) of the proof of~\ref{affine}.
Now $\theta'f(b)=0$ and $S_f(b)\neq 0$.
All the other algebraic indeterminates of $\theta'f$ are of strictly smaller rank, and so by induction when evaluated at $b$ they land in $F(\nabla_{t+N}a,B_tb)^{\alg}$.
On the other hand, the coefficients of $\theta'f$ can be written as fractions of $\Delta$-polynomials in $a$ over $F$ of order $\leq N+\ord(\theta')\leq N+t$ by choice of~$N$.
So $\theta'f(b)=0$ witnesses that $rb\in F(\nabla_{t+N}a,B_tb)^{\alg}$.
\end{proof}

\medskip
\subsection{Codimension one $\Delta$-subvarieties}
\label{subsection-cdb}
Suppose $X\subseteq K^n$ is an $F$-irreducible $\Delta$-variety.
We say that an $F$-irreducible $\Delta$-subvariety $Y\subseteq X$ is of {\em codimension one} if for generic $x\in X$ and $y\in Y$, $\trdeg_FF(\nabla_ty)=\trdeg_FF(\nabla_tx)-1$ for all sufficiently large $t$.
In this section we uniformly bound what is meant by ``sufficiently large".

\begin{proposition}
\label{boundcodim1}
Suppose $X\subseteq K^n$ is an irreducible $\Delta$-variety over $F$.
There exists $\ell\geq 0$ such that if $Y\subseteq X$ is a co-dimension one irreducible $\Delta$-subvariety over~$F$ then $\trdeg_FF(\nabla_ty)=\trdeg_FF(\nabla_tx)-1$ for all $t\geq\ell$, where $x\in X, y\in Y$ are generic.
\end{proposition}

\begin{proof}
Fix $c\in X$ generic and $B\subseteq\mathbb N^m\times\{1,\dots,n\}$ a transcendence index set for $c$ over~$F$.
So $\Lambda$ is a characteristic set for $I_\Delta(X)$, $E$ is the set of all $(e_1,\dots,e_m,j)\in\mathbb N^m\times\{1,\dots,n\}$ such that $\delta_1^{e_1}\dots\delta_m^{_m}u_j$ is a leader of an element of~$\Lambda$, and $B$ is the set of all points that do not lie above any element of $E$.
Fact~\ref{basis} tells us that $B_tc$ is a transcendence basis for $F(\nabla_tc)$ over $F$ for all $t<\omega$.

Fix a $\Delta$-subvariety $Y\subseteq X$ of codimension one, and generic $d\in Y$ over $F$.
We first argue that $\trdeg_FF(\nabla_td)\geq\trdeg_FF(\nabla_tc)-1$ for all $t\geq 0$.
We know that $F(\nabla_tc)$ is algebraic over $F(B_tc)$.
On the other hand,  $\nabla_td$ is a Zariski specialisation of $\nabla_tc$ since $Y\subseteq X$, and hence $F(\nabla_td)$ is algebraic over $F(B_td)$.
So if for some $t_0$ we had $\trdeg_FF(\nabla_{t_0}d)<\trdeg_FF(\nabla_{t_0}c)-1=|B_{t_0}|-1$, then there would be at least two elements of $B_{t_0}d$ that are algebraic over $F$ and the rest of the set.
As the $B_td$ form an increasing chain, this would persist and we would have that for all $t\geq t_0$, $\trdeg_FF(\nabla_{t}d)\leq |B_t|-2=\trdeg_FF(\nabla_{t}c)-2$, which contradicts the codimension one assumption.

Next, write $\Lambda=\{f_1,\dots,f_k\}$, listed as usual in strictly increasing order of rank and suppose $Y$ is such that there exists $g\in I_\Delta(Y)\setminus I_\Delta(X)$ with $\ord(g)\leq\ord(f_k)$.
Then setting $\ell_1:=\ord(f_k)$, which notice does not depend on $Y$, we have that~$g$ witnesses $I(\nabla_t d/F)\supsetneq I(\nabla_t c/F)$ for all $t\geq\ell_1$.
Hence $\trdeg_FF(\nabla_t d)\leq\trdeg_FF(\nabla_t c)-1$.

So it remains to consider those $Y$ such that $I_\Delta(Y)$ and $I_\Delta(X)$ agree up to order $\ord(f_k)$.
Let $Y$ be such and let $\Gamma=\{g_1,\dots,g_{k'}\}$ be a characteristic set for $I_\Delta(Y)$.
Then $\Gamma$ must have strictly lower rank than $\Lambda$ since $Y\subsetneq X$.
We claim that $k'>k$.
Indeed, if not, then there must be some $i< k$ such that $g_1,\dots,g_i$ have the same rank as $f_1,\dots,f_i$ while $\rank(g_{i+1})<\rank(f_{i+1})$.
From the way the ranking of $\Delta$-polynomials is defined this implies that $g_1,\dots,g_{i+1}$ all have order bounded by $\ord(f_k)$.
By our assumption on $Y$ it follows that $g_1,\dots,g_{i+1}\in I_\Delta(X)$, and so $\{g_1,\dots,g_{i+1}\}$ would be an autoreduced set in $I_\Delta(X)$ that is of strictly smaller rank than $\Lambda$, contradicting the minimality of characteristic sets.

Hence, it must be that case that $k'>k$ and that $g_1,\dots,g_k$ have the same rank as $f_1,\dots,f_k$.
But then the leaders of $\Gamma$ include all the leaders of $\Lambda$.
That is, if we set $E_Y$ to be all $(e_1,\dots,e_m,j)\in\mathbb N^m\times\{1,\dots,n\}$ such that $\delta_1^{e_1}\dots\delta_m^{_m}u_j$ is a leader of an element of~$\Gamma$, and set $B_Y$ to be the set of all points in $\mathbb N^m\times\{1,\dots,n\}$ that do not lie above any element of $E_Y$, then $B_Y$ is an initial subset of $B$.
Moreover, applying Fact~\ref{basis} to $Y$, we know that $(B_Y)_td$ is a transcendence basis for $F(\nabla_td)$ over $F$ for all $t\geq 0$, just as $B_tc$ is a transcendence basis for $F(\nabla_tc)$.
But now the codimension one hypothesis forces $B_Y=B\setminus\{r_Y\}$ for some $r_Y\in B$.
In particular, $r_Y$ has the special property that when you remove it from the initial set $B$ you still have an initial set.
A general study of the combinatorics of initial sets shows that an initial set can only have finitely many such points.
For example, in the terminology of~\cite{SitWell}, $r_Y$ is a {\em properly $0$-dimensional} subset of $B$ and Proposition~1 of~\cite{SitWell} says that an initial set can have only finitely many such.
%In fact, by Dickson's Lemma, every initial set $B'\subseteq \mathbb N^m \times \{1, \ldots , m\}$ is of the form $B'=\{r:r\not\geq e,\text{ for any }e\in E'\}$ for some finite $E'$.
So there exists $r_1,\dots,r_p\in B$, not depending on $Y$, such that $B_Y=B\setminus\{r_i\}$ for some $i=1,\dots,p$. 
Now setting $\ell_2:=\max\{|r_1|,\dots,|r_p|\}$ we have that for $t\geq \ell_2$,
\begin{eqnarray*}
\trdeg_FF(\nabla_t d)
&=&
|(B_Y)_t|\\
&=&
|\big(B\setminus\{r_i\}\big)_t| \ \ \ \ \ \ \ \ \ \ \text{for some $i=1,\dots,p$}\\
&=&
|B_t|-1  \ \ \ \ \ \ \ \ \ \ \ \ \ \ \ \text{ since $|r_i|\leq t$}\\
&=&
\trdeg_FF(\nabla_t c)-1
\end{eqnarray*}
So setting $\ell:=\max\{\ell_1,\ell_2\}$ proves Proposition~\ref{boundcodim1}.
\end{proof}

\begin{remark}
\label{determinecodim1}
Suppose $X\subseteq\mathbb A^n$ is an irreducible $\Delta$-variety over $F$.
Let $\ell$ witness the truth of Proposition~\ref{boundcodim1}.
Then all codimension one $\Delta$-subvarieties of $X$ are determined by their $\ell$th prolongations.
That is, given $Y, Z\subseteq X$ codimension one irreducible $\Delta$-subvarieties over~$F$ with $y\in Y$ and $z\in Z$ generic, if $\loc(\nabla_\ell y/F)=\loc(\nabla_\ell z/F)$ then $Y=Z$.
\end{remark}

\begin{proof}
Note that $Y$ is determined by the directed sequence of algebraic varieties $Y_t:=\loc(\nabla_ty/F)$, $t\geq 0$.
Suppose $Y_\ell=Z_\ell$ as codimension one algebraic subvarieties of $X_\ell$.
Given $t\geq \ell$, let $\pi:X_t\to X_\ell$ be the co-ordinate projection.
Since $Y_{t}$ is still a codimension one algebraic subvariety of $X_{t}$ by choice of $\ell$, it follows that the generic fibre of $Y_{t}$ over $Y_\ell$ is the full generic fibre of $\pi$.
Similarly for $Z_{t}$.
That is,  $Y_{t}$ and $Z_{t}$ are subvarieties of $X_{t}$ that project dominantly onto the same subvariety of $X_\ell$ with the same generic fibre -- by irreducibility they must agree.
Hence $Y=Z$, as desired.
\end{proof}

\medskip
\subsection{The finiteness theorem}
\label{subsection-ft}

\begin{theorem}
\label{ghj-dav}
Suppose $F\subseteq \mathcal C$ is a subfield of the constants and $X\subseteq K^n$ is an $F$-irreducible $\Delta$-variety.
Let $F\langle X\rangle$ denote the $\Delta$-rational function field of $X$.
If $F\langle X\rangle\cap\mathcal C=F$ then $X$ has only finitely many codimension one $F$-irreducible $\Delta$-subvarieties.
\end{theorem}

\begin{proof}
We first reduce to the case when $F=F^{\alg}$ and so $X$ is absolutely irreducible.
Suppose $c\in X$ is generic over $F$.
Then $c$ is a generic point of an irreducible component of $X$ over $F^{\alg}$.
If there is $b\in \big(F^{\alg}\langle c\rangle\cap\mathcal C\big)\setminus F^{\alg}$, then $b\in\acl(F,c)$ so that a canonical parameter for the finite orbit of $b$ over $Fc$, say $\bar b$, is a tuple from $F\langle c\rangle\cap\mathcal C$.
As $b\in\acl(\bar b)$, we must have that $\bar b$ is not defined over $F$.
That is, $\big(F\langle c\rangle\cap\mathcal C\big)\setminus F\neq\emptyset$.
Since $c$ is generic in $X$ over $F$, this contradicts the assumption on the the $\Delta$-constants of $F\langle X\rangle$.
Hence $F^{\alg}\langle c\rangle\cap\mathcal C=F^{\alg}$.
Assuming we have proved the theorem for irreducible $\Delta$-varieties, we would get that each irreducible component of $X$ has only finitely many irreducible codimension one $\Delta$-subvarieties over $F^{\alg}$.
But every irreducible component of an $F$-irreducible codimension one $\Delta$-subvariety of $X$ is a codimension one $\Delta$-subvariety of an irreducible component of $X$ over $F^{\alg}$.
So we obtain the desired finiteness statement for $X$ as well.
We may therefore assume that $X$ is irreducible and $F=F^{\alg}$.

Let $\ell$ be an upper bound for the order of all the elements of a fixed characteristic set of $I_\Delta(X)$, and also big enough to witness Proposition~\ref{boundcodim1}.
Let $c\in X$ be generic over $F$.
Set $v=\nabla_\ell c$, $V=\loc(v/F)$, and $S=\loc(\nabla v/F)\subseteq T^mV$, and $r=\trdeg\big(F(\nabla v)/F(v)\big)$.

We first claim that $S_v=\loc\big(\nabla v/F(v)\big)$ is an ($r$-dimensional) affine subspace of $(T_vV)^m$.
This follows from the fact that the $F$-transendence basis of $\nabla_{\ell+1}c$ given by Fact~\ref{basis} is also an $F(\nabla_\ell c)$-linear spanning set by Lemma~\ref{affine} applied to $t=\ell+1$.
In other words, there is a subtuple $(\eta_1,\dots,\eta_r)$ of $\nabla_{\ell+1}c\subseteq\nabla v$ that is a transcendence basis for $F(\nabla_{\ell+1}c)=F(\nabla v)$ over $F(\nabla_{\ell}c)=F(v)$, and such that $(1,\eta_1,\dots,\eta_r)$ is a linear basis for $F(\nabla v)$ over $F(v)$.
It follows that $S_v$ is an affine subspace of $(T_vV)^m$.

Let $V^\circ\subseteq V$ be a nonempty Zariski open subset of $V$ over $F$, such that $S_a$ is an affine subspace of $(T_aV)^m$ of dimension $r$, for all $a\in V^\circ$.
Let $S^\circ=S\upharpoonright_{V^\circ}$.
So $(V^\circ,S^\circ)$ is an algebraic $D$-variety of type $(m,r)$.

Now suppose, toward a contradiction, that $(V^\circ,S^\circ)$ admits a $D$-constant rational function $f\in F(V)\setminus F$.
So $d^mf$ vanishes on $S^\circ_v=S_v$ which contains $\nabla v$.
Hence $\nabla f(v)=d^mf\nabla(v)=(v,0)$, and so $\delta_kf(v)=0$ for all $k=1,\dots,m$.
Since $v=\nabla_{\ell}c$, we can view $f(v)\in F\langle X\rangle$, and we have just shown that it is a new $\Delta$-constant element of the $\Delta$-rational function field, contradicting the assumption on $X$.
Hence, $(V^\circ,S^\circ)$ admits no nonconstant $D$-constant rational functions.

By Theorem~\ref{ghjD}, $(V^\circ,S^\circ)$ has only finitely many codimension one $D$-subvarieties.

Let $Y\subset X$ be an irreducible codimension one $\Delta$-algebraic subvariety of $X$ over $F$, and let $y\in Y$ be generic.
By choice of $\ell$ witnessing Proposition~\ref{boundcodim1}, $Y_\ell=\loc(\nabla_\ell y/F)$ is a codimension one irreducible algebraic subvariety of $V$.
If $Y_\ell\cap V^\circ=\emptyset$, then $Y_\ell$ must be one of finitely many irreducible component of $V\setminus V^\circ$.
So assume $W:=Y_\ell\cap V^\circ\neq\emptyset$.
We claim that $W$ is a codimension one $D$-subvariety of $(V^\circ, S^\circ)$.
That is, $S\upharpoonright_W\subseteq T^mW$.
Indeed, from the fact that $y$ is a $\Delta$-specialisation of $c$, we get that $\nabla w$ is a Zariski specialisation of $\nabla v$, where $w:=\nabla_\ell y$.
So $\nabla w\in S_w$.
On the other hand,
\begin{eqnarray*}
\trdeg\big(\nabla w/F(w)\big)
&=&
\trdeg\big(\nabla_{\ell+1}y/F(\nabla_\ell y)\big)\\
&=&
\trdeg\big(\nabla_{\ell+1}c/F(\nabla_\ell c)\big) \ \ \ \ \ \text{by choice of $\ell$ witnessing~\ref{boundcodim1}}\\
&=&
r\\
&=&
\dim S_w
\end{eqnarray*}
It follows that $S_w=\loc\big(\nabla w/F(w)\big)$, and hence $S_w\subseteq (T_wW)^m$.
As $w$ is generic in $W$,
%\marginpar{\tiny Maybe this needs $S\upharpoonright_W$ irreducible? Actually not: the set of $w'$ such that $S_{w'} \subseteq (T_wW)^m$ is I suppose a closed condition, so true of generic of the irreducible $W$ must be true of all $W$.}
we get $S\upharpoonright_W\subseteq T^mW$, as desired.

Now, if $Y, Z\subseteq X$ are codimension one irreducible $\Delta$-subvarieties over~$F$, and $Y_\ell\cap V^\circ=Z_\ell\cap V^\circ\neq\emptyset$, then $Y_\ell=Z_\ell$, and so $Y=Z$ by  Remark~\ref{determinecodim1}.
So, from the fact that $(V^\circ,S^\circ)$ has only fnitely many codimension one $D$-subvarieties over $F$ we get that $X$ has only finitely many codimension one irreducible $\Delta$-subvarieties defined over $F$.
\end{proof}

\medskip
\subsection{A relative version and nonconstant coefficients}
The statement of Theorem~\ref{ghj-dav} can be improved so as to be independent of whether $F\langle X\rangle$ has new constants or not.

\begin{theorem}
\label{ghj-dav-absolute}
Suppose $F\subseteq \mathcal C$ is a constant subfield and $X\subseteq K^n$ is an $F$-irreducible $\Delta$-variety.
There exists an algebraic variety $V$ over~$F$ and a dominant $\Delta$-rational map $f:X\to V(\mathcal C)$ over $F$ such that all but finitely many codimension one $F$-irreducible $\Delta$-subvarieties of $X$ arise as $F$-irreducible components of $\Delta$-subvarieties of the form $f^{-1}\big(W(\mathcal C)\big)$ where $W\subseteq V$ is an algebraic subvariety over $F$.
\end{theorem}

\begin{proof}
The total constant field of $F\langle X\rangle$ is a function field over $F$ -- see for example~\cite[Proposition~14, $\S$II.11]{kolchin73}.
It is of the form $F(V)$ for some $F$-irreducible algebraic variety $V$.
Note that $L:=F\langle V(\mathcal C)\rangle=F(V)$, so that the inclusion $F(V)\subseteq F\langle X\rangle$ induces a dominant $\Delta$-rational map $f:X\to V(\mathcal C)$.

Over $L$ the $\Delta$-field $F\langle X\rangle$ has no new $\Delta$-constant elements.
This means that if $\eta\in V(\mathcal C)$ is generic over $F$, then $L=F(\eta)$ and the fibre $X_\eta:=f^{-1}(\eta)$ is an $L$-irreducible $\Delta$-subvariety of $X$
 with the property that its $\Delta$-rational function field over $L$ has $L$ as its constant field.
Applying Theorem~\ref{ghj-dav} to $X_\eta$, we get that $X_\eta$ has has only finitely many codimension one $L$-irreducible $\Delta$-subvarieties.

Suppose that $Y$ is an $F$-irreducible codimension one $\Delta$-subvariety of $X$ that maps dominantly onto $V(\mathcal C)$.
We claim that $Y_\eta$ is codimension one in $X_\eta$.
Indeed, let $c\in X_\eta$ and $d\in Y_\eta$ be generic over $L$, and let $\ell\geq 0$ be big enough so that the $\Delta$-rational map $f(u)$ is of the form $g(\nabla_\ell u)$ for some rational map $g(u)$.
Since $f(c)=f(d)=\eta$, we get that for all $t\geq\ell$, $\eta\in F(\nabla_tc)$ and $\eta\in F(\nabla_td)$.
Hence
$$\trdeg_FF(\nabla_tc)= \trdeg_LL(\nabla_tc)+\trdeg_FL$$
and
$$\trdeg_FF(\nabla_td)= \trdeg_LL(\nabla_td)+\trdeg_FL$$
Taking $\ell$ larger, we may also assume that $\trdeg_FF(\nabla_td)=\trdeg_FF(\nabla_tc)-1$, for all $t\geq\ell$.
So $\trdeg_LL(\nabla_td)=\trdeg_LL(\nabla_tc)-1$, for all $t\geq\ell$, as desired.

We have proved that $X$ has only finitely many codimension one $F$-irreducible $\Delta$-subvarieties that map dominantly onto $V(\mathcal C)$.
So it remains to consider those that either fall in the indeterminacy locus of $f$, or get mapped dominantly onto proper $\Delta$-subvarieties of $V(\mathcal C)$.
Since codimension one $F$-irreducible $\Delta$-subvarieties are maximal among proper $F$-irreducible $\Delta$-subvarieties, those that land in the indeterminacy locus of $f$ must be $F$-irreducible components of this indeterminacy locus; and hence there are only finitely many of them.
Finally, suppose $Y\subseteq X$ is a codimension one $F$-irreducible $\Delta$-subvarieties such that the Kolchin closure of $f(Y)$ is of the form $W(\mathcal C)$ for some proper irreducible algebraic subvariety $W\subseteq V$ over $F$.
Then, by maximality of $Y$ in $X$, $Y$ is an irreducible component of $f^{-1}\big(W(\mathcal C)\big)$.
\end{proof}

Note that Theorem~\ref{ghj-dav} is a special case of Theorem~\ref{ghj-dav-absolute}: under the assumptions on $X$ imposed by~\ref{ghj-dav}, the map $f$ given by~\ref{ghj-dav-absolute} would have to be constant, and so the conclusion would be that there are only finitely many codimension one $F$-irreducible $\Delta$-subvarieties.

\begin{remark}
The $\Delta$-rational map $f:X\to V(\mathcal C)$ that we constructed in the above proof could be called an {\em algebraic reduction} of $X$, in analogy to complex bimeromorphic geometry, and will satisfy a certain natural universal property that we leave to the reader to formulate.
\end{remark}

One advantage of this latter formulation is that it generalises readily to $\Delta$-varieties not necessarily defined over the constants.

\begin{corollary}
\label{ghj-dav-general}
Suppose $F\subseteq\mathcal C$ is a constant subfield, and $L$ is a finitely generated $\Delta$-field extension of $F$.
Let $X\subseteq K^n$ be an $L$-irreducible $\Delta$-variety.
There exists an algebraic variety $V$ over the constants of $L$, $L_0$, and a dominant $\Delta$-rational map $f:X\to V(\mathcal C)$ over $L$, such that all but finitely many codimension one $L$-irreducible $\Delta$-subvarieties of $X$ arise as $L$-irreducible components of $\Delta$-subvarieties of the form $f^{-1}\big(W(\mathcal C)\big)$ where $W\subseteq V$ is an algebraic subvariety over $L_0$.

In particular, if the $\Delta$-constant field of $L\langle X\rangle$ is contained in $L$ then $X$ has only finitely many codimension one $L$-irreducible $\Delta$-subvarieties.
\end{corollary}

\begin{proof}
The ``in particular" clause follows from the main statement in exactly the same way that Theorem~\ref{ghj-dav} is a special case of Theorem~\ref{ghj-dav-absolute}.

Let $L=F\langle a\rangle$, $Z=\kloc(a/F)$, $b\in X$ generic over $L$, $\widehat X$ the $\kloc(a,b/F)$, and $\pi:\widehat X\to Z$ the co-ordinate projection taking $(a,b)$ to $a$.
Then $X$ can be identified with the generic fibre $\widehat X_a$ of $\pi$.
Let $\widehat V$ be the algebraic variety over $F$, and $\widehat f:\widehat X\to\widehat V(\mathcal C)$ the dominant $\Delta$-rational map over~$F$, given by Theorem~\ref{ghj-dav-absolute} applied to $\widehat X$.
Note that the Kolchin closure of $\widehat f(X)$ is, by stable embedability of the constants, of the form $V(\mathcal C)$ for some algebraic subvariety $V\subseteq \widehat V$ defined over $L_0$.
Restricting to the generic fibre of $\pi$ we get a dominant $\Delta$-rational map $f:X\to V(\mathcal C)$.

Now suppose $Y\subseteq X$ is a codimension one $L$-irreducible $\Delta$-subvariety over $L$.
Let $c\in Y$ be generic over $L$ and set $\widehat Y:=\kloc(a,c/F)$, so that $Y=\widehat Y_a$.
We claim that $\widehat Y$ is of codimension one in $\widehat X$.
Indeed, let $N$ be big enough to witness Lemma~\ref{additivity} applied to $(a,b)$ and $(a,c)$.
That is, if $B$ is a transcendence index set for $b$ over $L$ and $C$ is a transcendence index set for $c$ over~$L$, then $\nabla_tb\subseteq F(\nabla_{t+N}a,B_tb)^{\alg}$ and $\nabla_tc\subseteq F(\nabla_{t+N}a,C_tc)^{\alg}$ for all $t\geq 0$.
So $\trdeg\big(\nabla_tb/F(\nabla_{t+N}a)\big)=|B_t|$ and $\trdeg\big(\nabla_tc/F(\nabla_{t+N}a)\big)=|C_t|$.
Hence
\begin{eqnarray*}
\trdeg_FF\big(\nabla_t(\nabla_Na,b)\big)
&=&
|B_t|+\trdeg_FF\big(\nabla_t(\nabla_Na)\big)\\
&=&
\trdeg_LF\big(\nabla_t(b)\big)+\trdeg_FF\big(\nabla_t(\nabla_Na)\big)\\
&=&
\trdeg_LF\big(\nabla_t(c)\big)+1+\trdeg_FF\big(\nabla_t(\nabla_Na)\big)\\
&=&
|C_t|+\trdeg_FF\big(\nabla_t(\nabla_Na)\big)+1\\
&=&
\trdeg_FF\big(\nabla_t(\nabla_Na,c)\big)+1
\end{eqnarray*}
where the third equality is for sufficiently large $t$, as $Y$ is codimension one in $X$.
So $\kloc(\nabla_Na,c/F)$ has codimension one in $\kloc(\nabla_Na, b/L)$.
Applying a $\Delta$-isomorphism we get that $\widehat Y$ has codimension one in $\widehat X$.

Distinct $Y$ will give rise to distinct $\widehat Y$, so for all but finitely many $Y$ we will get that $\widehat Y$ is an $F$-irreducible component of $\widehat f^{-1}\big(\widehat W(\mathcal C)\big)$ for some algebraic subvariety $\widehat W\subseteq \widehat V$ defined over $F$.
Restricting to the generic fibre of $\pi$, $Y$ is an $L$-irreducible component of $\widehat f^{-1}\big(\widehat W(\mathcal C)\big)\cap X=f^{-1}\big(\widehat W\cap V(\mathcal C)\big)$.
So $W:=\widehat W\cap V$ works.
\end{proof}

The assumption that $L$ be finitely generated over its constants is necessary.
There exist $\Delta$-varieties of order $1$ that admit no nonconstant $\Delta$-rational functions to $\mathcal C$ over any parameter set (i.e., whose Kolchin generic type is orthogonal to the constants) -- for example, a general way of producing these (with $m=1$) was developed in~\cite[$\S$2]{hrushovski-itai}.
Let $X$ be such, fix an infinite collection $P$ of points on $X$, and then pass to a $\Delta$-field extension $L$ over which $X$ and all the points in $P$ are defined.
So the $\Delta$-constants of $L\langle X\rangle$ are contained in $L$, but each member of $P$ is a codimension one $\Delta$-subvariety over $L$.

\bigskip
\section{Applications}
\label{section-applications}

\noindent
We give three applications of Corollary~\ref{ghj-dav-general}.
The first two are obtained simply by replacing, in known arguments, the use of Hrushovski's~\cite[Proposition~2.3]{hrushovski-jouanolou} by our extension to the partial and nonconstant coefficient setting.
The third application, on bounding the height of algebraic solutions to first-order differential equations over $\CC(t)$, seems to not have been noticed before and makes essential use of our relative formulation.

\medskip
\subsection{Lascar and Morley rank agree in dimension two}
We continue to work in a sufficiently saturated model $(K,\Delta)\models\operatorname{DCF}_{0,m}$ with field of total constants $\mathcal C$, and over a small $\Delta$-field of definition $F\subseteq K$.

Let us first explain what ``dimension" we have in mind.
If $X\subseteq K^n$ is an $F$-irreducible $\Delta$-variety with $\Delta$-rational function field $F\langle X\rangle$ of finite transcendence degree over $F$, then we say that $X$ is of {\em finite dimension} and we call $\trdeg_FF\langle X\rangle$ the {\em dimension of $X$}.
Note that $X$ being of finite dimension~$r$ is equivalent to the dimension function of $X$, as defined in $\S$\ref{subsection-c1dav} above, being eventually of constant value $r$.
We extend this terminology to $F$-definable sets $S\subseteq K^n$, by saying that $S$ is of dimension $r$ if all the $F$-irreducible components of the Kolchin closure of $S$ are of finite dimension and the maximum of those dimensions is $r$.

In general Lascar and Morley rank do not agree in differentially closed fields; a counterexample of dimension five was constructed by Hrushovski and Scanlon~\cite{HrushovskiScanlon1999}.
However, it was noted by Marker and Pillay that these ranks do agree for $0$-definable sets of dimension two.
(If the dimension is one, then so are the Lascar and Morley ranks.)
Their argument, which was communicated to us by David Marker, used Hrushovski's theorem on hypersurfaces of differential algebraic varieties over the constants.
Given our extension of this theorem to nonconstant coefficient fields, the Marker-Pillay argument now shows that Lascar and Morley rank agree on arbitrary definable sets of dimension two.
We give the proof here, for the sake of completeness.

\begin{theorem}
\label{lascarmorley}
Suppose $S$ is a definable set of dimension two.
Then the Morley and Lascar ranks of $S$ agree.
\end{theorem}

\begin{proof}
Taking irreducible components of Kolchin closures it suffices to prove the theorem for $S=X\subseteq K^n$ an irreducible $\Delta$-variety.
Since Lascar rank is bounded by Morley rank which is bounded by the dimension, the only case we have to consider is when the Morley rank of $X$ is two.

Let $F$ be a finitely generated $\Delta$-field over which $X$ is defined.
It suffices to prove the existence of an infinite $\Delta$-subvariety of $X$ that is not defined over~$F^{\alg}$.
Indeed, let $Y\subseteq X$ be such.
We can further assume that Y is irreducible and defined over some $\Delta$-field extension $F'\supseteq F$.
Let $d\in Y$ be Kolchin generic over $F'$.
If $\tp(d/F')$ were a nonforking extension of $\tp(d/F)$ then $Y$ would be an irreducible component of $\kloc(d/F)$, contradicting the assumption that $Y$ is not defined over $F^{\alg}$.
Hence, $\tp(d/F')$ is a nonalgebraic forking extension of $\tp(d/F)$, proving that the latter is of Lascar rank at least two.
Hence $X$ would be of Lascar rank two.

Since $X$ is of Morley rank two it has infinitely many infinite $\Delta$-subvarieties, say $(Y_i:i<\omega)$.
If any of these are not defined over $F^{\alg}$ then we are done by the previous paragraph, so we may assume they are all defined over $F^{\alg}$.
Replacing $Y_i$ by the union of its $F$-conjugates, we may assume that each $Y_i$ is defined over~$F$.
Moreover, taking irreducible components, we may assume that they are all $F$-irreducible.
Since $X$ is of dimension two, and each $Y_i$ is a proper infinite $\Delta$-variety, each $Y_i$ must be of codimension one in $X$.
By Corollary~\ref{ghj-dav-general}, there is a dominant $\Delta$-rational map $f:X\to\mathbb A^1(\mathcal C)$.
The generic fibre of $f$ will be an infinite $\Delta$-subvariety of $X$ that is not defined over~$F^{\alg}$.
So we are done by the previous paragraph.
\end{proof}

Hrushovski and Scanlon asked in~\cite{HrushovskiScanlon1999} for an explanation of the gap between dimensions two and five, and indeed, as far as we know, it is still not known what happens in dimensions three and four.\footnote{A putative three-dimensional example where Lascar and Morley rank differ was given in~\cite{Nagloo2011algebraic}, but the computations there seem to be incorrect.}

\medskip
\subsection{Dimension one strongly minimal sets}
Hrushovski's motivation in~\cite{hrushovski-jouanolou} for considering the differential-algebraic geometric consequences of Jouanlolou's theorem was to understand the structure of strongly minimal sets of dimension one in $\operatorname{DCF}_0$.
He shows that they are either nonorthogonal to the constants or $\aleph_0$-categorical.
Having extended the differential-algebraic geometric results to the partial case, we follow~\cite[Corollaries 2.5 and 2.6]{hrushovski-jouanolou} to obtain an analogous result for $\operatorname{DCF}_{0,m}$.

\begin{theorem}
\label{aleph0cat}
Suppose $S$ is a strongly minimal dimension one definable set that is orthogonal to $\mathcal C$.
Then $S$ is $\aleph_0$-categorical.
\end{theorem}

\begin{proof}
Since $S$ is strongly minimal, to deduce $\aleph_0$-categoricity it suffices to prove that for every finite set $B$ over which $S$ is defined, $\acl(B)\cap S$ is finite.

The Kolchin closure of $S$ has a unique infinite irreducible component, say $X$.
Let $F$ be a finitely generated $\Delta$-field over which $X$ is defined and such that $B\subseteq F$.
All but finitely many points of $S$ are in $X$, so it suffices to show that $\acl(B)\cap X$ is finite.
Let $a\in\acl(B)\cap X$.
Then $a\in X(F^{\alg})$ and so $Y:=\kloc(a/F)$ is a finite $F$-irreducible $\Delta$-subvariety $X$.
Now $\trdeg_FF\langle X\rangle=1$ by the dimension one assumption on $S$.
Since $Y$ is finite it is of codimension one.
So $\acl(B)\cap X$ is contained in the union of all codimension one $F$-irreducible $\Delta$-subvarieties of $X$.
Since $S$ is orthogonal to~$\mathcal C$, $X$ admits no nonconstant $\Delta$-rational maps over $F$ to~$\mathcal C$.
Corollary~\ref{ghj-dav-general} therefore implies that $X$ has only finitely many codimension one $F$-irreducible $\Delta$-subvarieties.
Since all such subvarieties must be finite, their union is a finite subset of $X$.
\end{proof}

It is well known that the theorem fails for strongly minimal sets of higher finite dimension.
Manin kernels appear as strongly minimal groups that are orthogonal to the constants.
For some time it was open whether all strongly minimal sets with {\em trivial} pregeometry in $\operatorname{DCF}_0$ were $\aleph_0$-categorical, but the first author and Thomas Scanlon~\cite{jfunction} have shown recently that the $j$-function gives rise to counterexamples in dimension three.

\medskip
\subsection{Algebraic solutions to first-order differential equations}
In~\cite{Erem}, Eremenko proves that if $P\in\CC(t)[x,y]$ is a nonzero polynomial in two variables over the field of rational functions, then there is a constant $N=N(P)$ such that all solutions in $\big( \CC(t),\frac{d}{dt}\big)$ to the differential equation $P(x,x')=0$ are of degree bounded by $N$.
Here the degree of a rational function is the maximum of the degrees of the numerator and denominator of $g$ when expressed as a ratio of coprime polynomials.
He suggests that ``it is a challenging unsolved question whether [the above result] can be extended to {\em algebraic} solutions," that is to solutions in $\big(\CC(t)^{\alg},\frac{d}{dt}\big)$.
We give here such an extension.

In order to state the extension we need to make sense of the ``degree" of an element of $\CC(t)^{\alg}$.
The natural thing to consider is the function field absolute logarithmic height, which we now recall and details of which can be found in~\cite[Chapters~3 and~4]{lang}.
Given $g\in\CC(t)^{\alg}$ let $k$ be a finite extension of $\CC(t)$ in which $g$ lies.
Writing $k=\CC(E)$ for some smooth projective curve $E$, we view $g$ as a rational function on $E$, and the {\em height} $h(g)$ is defined to be the degree of the polar divisor of~$g$ -- so the number of poles of $g$ counting multiplicity -- divided by $[k:\CC(t)]$.
This quantity does not depend on the choices of $k$ and $E$ made.
Note that $h$ on $\CC(t)^{\alg}$ extends degree on $\CC(t)$.

\begin{theorem}
\label{boundheight}
Suppose $P\in\CC(t)[x,y]$ is nonzero.
There exists $N=N(P)\in \mathbb N$, such that all solutions to $P(x,x')=0$ in $\big(\CC(t)^{\alg},\frac{d}{dt}\big)$ are of height $\leq N$.
\end{theorem}

\begin{proof}
We work in a saturated model $(K,\delta)\models\operatorname{DCF}_0$ extending $\big( \CC(t),\frac{d}{dt}\big)$ and with field of constants $\mathcal C$.
We may assume that $P$ is irreducible and of positive degree in both $x$ and $y$.
Let $E\subseteq\mathbb A^2$ be the algebraic curve over $\CC(t)$ defined by $P(x,y)=0$, and let $X:=\{(a_1,a_2)\in E(K):\delta(a_1)=a_2\}$.

By Corollary~\ref{ghj-dav-general}, there is an algebraic variety $V\subseteq\mathbb A^n$ defined over $\CC$ and a dominant $\delta$-rational map $f:X\to V(\mathcal C)$ over $\CC(t)$ such that all but finitely many $\CC(t)$-irreducible $\delta$-subvarieties of $X$ of codimension one are irreducible components of sets of the form $f^{-1}\big(W(\CC)\big)$ where $W$ is a $\CC$-definable algebraic subvariety of~$V$.
Since $X$ is a one-dimensional $\delta$-variety, $V(\mathcal C)$ is of dimension $\leq 1$ as a $\delta$-variety, and hence as an algebraic variety we have $\dim V\leq 1$.
If $a\in X\big(\CC(t)^{\alg}\big)$ then $\kloc\big(a/\CC(t)\big)$ is a finite $\CC(t)$-irreducible $\delta$-subvariety of~$X$, and hence of codimension one.
So if $\dim V=0$ then $X(\CC(t)^{\alg})$ is finite, and the theorem follows vacuously.
We may therefore assume that $\dim V =1$.
As the only proper $\CC$-definable algebraic subvarieties of $V$ are its $\CC$-points, we conclude that all but finitely many $\CC(t)^{\alg}$-points of $X$ get mapped by $f$ to $V(\CC)$.
Since the height function is zero on $V(\CC)$, our strategy now is to use $f$ to bound the height function on $X\big(\CC(t)^{\alg}\big)$.

First, we claim that the $\delta$-rational map $f$ extends to a rational map on $E$.
Let $X_0$ be obtained from $X$ by removing the (finite) set of points where the partial derivative $P_y:=\frac{\partial}{\partial y}P$ vanishes.
If $a=(a_1,a_2)\in X_0$ then
\begin{eqnarray*}
\delta(a_1)&=&a_2\\
\delta(a_2)&=&-\frac{P_x(a_1,a_2)a_2}{P_y(a_1,a_2)}-P^{\delta}(a_1,a_2)
\end{eqnarray*}
That is, $\delta$ agrees with the rational map $(y, -\frac{P_xy}{P_y}-P^{\delta})$ on $X_0$.
Replacing occurrences of $\delta$ in $f$ by this rational map, we obtain a $\CC(t)$-definable rational map $\alpha$ that agrees with $f$ on $X_0$.
As $X_0$ is Zariski dense in $E$, we have that $\alpha:E\to V$.

The height function defined above extends to $\CC(t)^{\alg}$-points of $E$ and $V$.
First, on any projective space the function field absolute logarithmic height for $\CC(t)^{\alg}$-points is defined as follows:
If $g=(g_0:\dots:g_\ell)\in\mathbb P^\ell(k)$ where $k$ is a finite extension of $\CC(t)$, and writing $k=\CC(E)$ for some smooth projective curve $E$, then $h(g)$ is the degree of the supremum of the polar divisors of $g_0,\dots,g_\ell$ on $E$, divided by $[k:\CC(t)]$.
This height agrees with the height defined earlier on $\CC(t)^{\alg}$ under the identification of $g\in \CC(t)^{\alg}$ with $(1:g)\in \mathbb P^1\big(\mathbb C(t)^{\alg}\big)$.
See~\cite[$\S$3.3]{lang} for more details.
Now, embed $E$ in $\mathbb P^2$ by identifying $(a_1,a_2)$ with $(1:a_1:a_2)$, and denote by $\overline E$ the Zariski closure of $E$ in $\PP^2$.
We thus have a height function on $\overline E\big(\CC(t)^{\alg}\big)$ coming from $\PP^2$.
Similarly, let $\overline V$ be the projective closure of $V$ in $\mathbb P^n$, and denote again by $h$ the corresponding height function on $\overline V\big(\CC(t)^{\alg}\big)$.
Note that the height of a $\CC$-point is zero.

Consider the rational map $\alpha:\overline E\to \overline V$.
Resolving the singularities of the graph of~$\alpha$, we have a smooth projective $\CC(t)$-definable curve $\Gamma$ with surjective morphisms $\pi_E:\Gamma\to\overline E$ and $\pi_V:\Gamma\to\overline V$ such that $\alpha\circ\pi_E =\pi_V$ on a cofinite subset of $\Gamma$.
Let $h_E:=h\circ\pi_E$ and $h_V:=h\circ\pi_V$ be the height functions on $\Gamma\big(\mathbb C(t)^{\alg}\big)$ induced by these maps.
By the functoriality of Weil's height machine, see~\cite[$\S$4.1 and~$\S$4.2]{lang}, up to equivalence, these heights depend only on the divisors of the linear systems on $\Gamma$ associated to $\pi_E:\Gamma\to \mathbb P^2$ and $\pi_V:\Gamma\to \mathbb P^n$ respectively, and not on the morphisms themselves.
Here two positive real-valued functions are said to be {\em equivalent} if their difference is a bounded function.
Moreover, by~\cite[Corollary~4.3.5]{lang}, which is  the algebraic equivalence property of Weil's height machine in the case of curves, $h_E$ is ``quasi-equivalent" to $rh_V$, where $r$ is the ratio of the degrees of the corresponding divisors.
{\em Quasi-equivalence} means that for every $\epsilon>0$ there are positive constants $c_2,c_2$ such that  $(1-\epsilon)rh_V-c_1\leq h_E\leq (1+\epsilon)rh_V+ c_2$.
Now, for all but finitely many $a\in X\big(\mathbb C(t)^{\alg}\big)$, we know that $\alpha(a)=f(a)\in V(\CC)$, and hence $h(\alpha(a))=0$.
With possibly finitely many more exceptions, we also have $b\in \Gamma\big(\mathbb C(t)^{\alg}\big)$ such that $\pi_E(b)=a$ and $\pi_V(b)=\alpha(a)$.
Hence, for such $a$ we get
$$h(a)=h_E(b)\leq (1+\epsilon)rh_V(b)+c_2=(1+\epsilon)rh(\alpha(a))+c_2=c_2$$
It follows that there is a uniform bound on the height of all points in $X\big(\mathbb C(t)^{\alg}\big)$.

If $g\in\CC(t)^{\alg}$ is a solution to $P(x,x')=0$, then $(g,\delta g)\in X\big(\mathbb C(t)^{\alg}\big)$, and from the way the heights were defined, $h(g)\leq h(g,\delta g)$.
So we have shown that a uniform bound exists on the height of all algebraic solutions to $P(x,x')=0$.
\end{proof}

We have restricted our attention above to the ordinary case for the sake of concreteness, and because it was in this form that the problem is mentioned in~\cite{Erem}.
However, since the setting of Corollary~\ref{ghj-dav-general} is after all that of partial differentiation, the above arguments extend to the partial case.
One obtains the following statement, which we leave to the reader to verify:
{\em Suppose $L=\CC(t_1,\dots, t_m)$ is the field of rational functions in $m$ variables, and $E\subseteq\mathbb A^{m+1}$ is an algebraic curve over~$L$.
Then
$$\left\{g\in L^{\alg}:(g,\frac{\partial g}{\partial t_1},\frac{\partial g}{\partial t_2},\dots,\frac{\partial g}{\partial t_m})\in E\right\}$$
is of bounded height.}
Here we take the absolute logarithmic height corresponding to the function field $L/\mathbb C$.
Note also that the complex numbers play no special role, the result remains true over any field of characteristic zero.

\appendix

\bigskip
\section{Two lemmas in exterior algebra}

\noindent
The following two straightforward linear algebra lemmas that are used in the proof of the Jouanolou-Hrushovski-Ghys theorem appear in Hrushovski's unpublished manuscript~\cite{hrushovski-jouanolou}.
As we could not find a good published reference we reproduce them here almost verbatum.

\begin{lemma}[Hrushovski~\cite{hrushovski-jouanolou}]
\label{udi-lemma1.5}
Suppose $k\subseteq K$ are fields, $V$ is a $K$-vector space, and $U\subseteq V$ is a $k$-subspace of $V$.
Working in the exterior powers of $V$ over $K$, suppose there exists $\ell\geq 1$ such that $\dim_k\span_k\{u_1\wedge\dots\wedge u_\ell:u_1,\dots,u_\ell\in U\}$ is finite and greater than zero.
Then $\dim_kU$ is finite.
\end{lemma}

\begin{proof}
Let $B=\span_k\{u_1\wedge\dots\wedge u_\ell:u_1,\dots,u_\ell\in U\}$ and choose some nonzero $\beta:=u_1\wedge\dots\wedge u_\ell$ with $u_1,\dots,u_\ell\in U$.
Consider the $k$-linear map $K\to \bigwedge^\ell V$ given by $a\mapsto a\beta$ and let $A$ be the preimage of $B$.
So $A$ is a finite dimensional $k$-vector subspace of $K$.

We claim that $\dim_k\big(U\cap\span_K\{u_1,\dots,u_{\ell-1}\}\big)$ is finite.
Indeed, if $v=\displaystyle\sum_{i=1}^{\ell-1}a_iu_i$ is in $U$ then for all $i\leq\ell-1$ we have
\begin{eqnarray*}
B\ \ni\ u_1\wedge\dots\wedge u_{i-1}\wedge v\wedge u_{i+1}\wedge\dots\wedge u_\ell
&=&
a_i\beta
\end{eqnarray*}
so that $a_i\in A$.
It follows that $\displaystyle U\cap\span_K\{u_1,\dots,u_{\ell-1}\}\subseteq\sum_{i=1}^{\ell-1}Au_i$, and hence is finite dimensional over $k$ as $A$ is.

Now consider the $k$-linear map $U\to B$ given by $v\mapsto u_1\wedge\dots\wedge u_{\ell-1}\wedge v$.
Since $u_1\wedge\dots\wedge u_{\ell-1}\neq 0$ the kernel of this map is $U\cap\span_K\{u_1,\dots,u_{\ell-1}\}$.
As both the kernel and the image are finite dimensional $k$-vector spaces, so is $U$.
\end{proof}

\begin{lemma}[Hrushovski~\cite{hrushovski-jouanolou}]
\label{udi-lemma1.4}
Let $K$ be a field, $V$ a $K$-vector space, and $V^*$ its dual.
Suppose $\alpha_1,\dots,\alpha_\ell\in V^*$ are such that $\gamma:=\alpha_1\wedge\dots\wedge\alpha_\ell\neq 0$, and $\omega$ is another wedge product of elements of $V^*$ such that $\omega\wedge\alpha_i=0$ for all $i=1,\dots,\ell$.
Then for any $\beta\in V^*$, if $\beta\wedge \gamma=0$ then $\beta\wedge\omega=0$.
\end{lemma}

\begin{proof}
We may assume $\omega\neq 0$.
Consider the $K$-subspace
$$W:=\{\alpha\in V^*:\alpha\wedge\omega=0\}.$$
If we write $\omega=\beta_1\wedge\dots\wedge\beta_p$, then certainly each $\beta_i\in W$.
As $\omega$ is nonzero the $\beta_i$ are linearly independent.
On the other hand, if $\alpha\wedge\omega=0$ then $\alpha\in\span_K\{\beta_1,\dots,\beta_p\}$, so that $\{\beta_1,\dots,\beta_p\}$ is a basis for $W$.

On the other hand each $\alpha_i\in W$ by assumption, and as $\gamma$ is nonzero these too are linearly independent.
Extend to another basis $\{\alpha_1,\dots,\alpha_\ell,\alpha_{\ell+1},\dots\alpha_p\}$.
Then $\beta_1\wedge\dots\wedge\beta_p=a\alpha_1\wedge\dots\wedge\alpha_p$ for some $a\in K$, and so
$\omega=a\gamma\wedge\alpha_{\ell+1}\wedge\dots\wedge\alpha_p$.
From this it is clear that if $\beta\wedge \gamma=0$ then $\beta\wedge\omega=0$.
\end{proof}

%\bibliographystyle{plain}
%\bibliography{jouanolou}

\end{document}